\documentclass[11 pt]{article}

\usepackage{amsmath}
\usepackage{amsthm}
\usepackage[linesnumbered,ruled,vlined]{algorithm2e}
\usepackage{xcolor}
\usepackage{dirtytalk}

\newtheorem{theorem}{Theorem}[section]
\newtheorem{lemma}[theorem]{Lemma}

\newtheorem{definition}{Definition}[section]
\newtheorem{assumption}[theorem]{Assumption}
\newtheorem{proposition}[theorem]{Proposition}

\theoremstyle{remark}
\newtheorem*{remark}{Remark}

\newcommand{\D}{\mathcal{D}}

\newcommand{\E}{\mathbb{E}}

\newcommand{\ra}{\rightarrow}
\newcommand{\R}{\mathbb{R}}
\newcommand{\Ra}{\Rightarrow}

\newcommand{\sphere}{\mathcal{S}}

\newcommand{\Unif}{\textsf{Unif}}

\newcommand{\vol}{\text{vol}}
\newcommand{\W}{\mathcal{W}}
\newcommand{\x}{x}
\newcommand{\y}{y}
\newcommand{\X}{\mathcal{X}}
\newcommand{\Y}{\mathcal{Y}}

\usepackage[utf8]{inputenc} 
\usepackage[T1]{fontenc}    
\usepackage{hyperref}       
\usepackage{url}            
\usepackage{booktabs}       
\usepackage{amsfonts}       
\usepackage{nicefrac}       
\usepackage{microtype}      
\usepackage{xcolor}         
\usepackage{defs} 
\usepackage[normalem]{ulem}

\title{Zeroth-Order Methods for Convex-Concave Minmax Problems: Applications to Decision-Dependent Risk Minimization}

\author{%
   Chinmay Maheshwari${}^{1}$\thanks{Equal contribution
}, Chih-Yuan Chiu${}^{1*}$, Eric Mazumdar$^1$,\\ {S. Shankar Sastry}$^1$, {Lillian J. Ratliff}$^2$\\
$^{1}$Electrical Engineering and Computer Sciences, University of California, Berkeley\\
$^2$Electrical and Computer Engineering, University of Washington, Seattle

}
\date{}

\allowdisplaybreaks

\newcommand{\ljr}[1]{\textcolor{Blue}{LJR: #1}}

\begin{document}

\maketitle

\begin{abstract}
  Min-max optimization is emerging as a key framework for analyzing problems of robustness to strategically and adversarially generated data. We propose a random reshuffling-based gradient free Optimistic Gradient Descent-Ascent algorithm for solving convex-concave min-max problems with finite sum structure. 
  We prove that the algorithm enjoys the same convergence rate as that of zeroth-order algorithms for convex minimization problems. We further specialize the algorithm to solve distributionally robust, decision-dependent learning problems, where gradient information is not readily available. Through illustrative simulations, we observe that our proposed approach learns models that are simultaneously robust against adversarial distribution shifts and strategic decisions from the data sources, and outperforms existing methods from the strategic classification literature.
  
\end{abstract}

\section{Introduction}
\label{sec: Intro}




The deployment of learning algorithms in real-world scenarios necessitates versatile and robust algorithms that operate efficiently under mild information structures. {Min-max optimization has been used as a tool ensure robustness in variety of domains e.g. robust optimization \cite{ben2009robust}, robust control \cite{hast2013pid}, to name a few. } Recently, min-max optimization has emerged as a promising framework for framing problems of algorithmic robustness against adversaries \cite{Goodfellow2014GANs,Steinhardt2017CertifiedDefense,Madry2017TowardsDeepLearningModelsResistanttoAdversarialAttacks}, strategically generated data \cite{Dong2018StrategicClassification, brown2020performative}, and distributional shifts in dynamic environments \cite{YuMazumdar2021FastDistributionallyRobustLearning}.

Despite this, recent works in machine learning and robust optimization on designing and analyzing stochastic algorithms for min-max optimization problems have largely operated on a number of assumptions that preclude their application to a broad range of real-world problems e.g., access to first-order oracles that provide exact gradients \cite{Yang2020GlobalConvergenceandVarianceReduction, Nouiehed2019SolvingAClassOfNonconvexMinMaxGames, Jin2020WhatIsLocalOptimality} or restrictive structural assumptions such as strong convexity \cite{Liu2019MinMaxWithoutGradients, Wang2020ZerothOrderAlgorithmsforNonconvexMinimaxProblems, Sadiev2020ZerothOrderAlgorithmsforSmoothSaddlePointProblems}. Moreover, the developed theory is often not well-aligned with the practical implementation of these algorithms in real-world machine learning applications. { For example, \cite{Beznosikov2020GradientFreeMethodsWithInexactOracle}
 propose zeroth-order methods for convex-concave problems but the proposed algorithm may not be suitable for machine learning applications where the objective function is a sum of large numbers of component functions (depending on the size of dataset). Indeed, in order to compute the gradient estimate at any iteration Beznosikov et al requires perturbing \textbf{all} the  functions which might not be suitable/possible for many applications.}
Furthermore, stochastic gradient methods are often used with random reshuffling (without replacement) in practice, yet their theoretical performance is usually characterized under the assumption of uniform sampling with replacement \cite{Bottou2009CuriouslyFastConvergenceOfSomeSGDAlgorithms, JainNagaraj2019SGDWithoutReplacement}.

 In this work, we do away with these assumptions and formulate a gradient-free (zeroth-order), random reshuffling-based algorithm with non-asymptotic convergence guarantees under mild structural assumptions on the underlying min-max objective. 
 {Our convergence guarantees are established by balancing the bias and variance of the zeroth-order gradient estimator \cite{Bravo2018BanditLearning}, using coupling-based arguments to analyze the correlations between iterates due to the random reshuffling procedure \cite{JainNagaraj2019SGDWithoutReplacement}, and exploiting the recent connections between the Optimistic Gradient Descent Ascent (OGDA) and Proximal Gradient algorithms \cite{MokhtariOzdaglar2020ConvergenceRO}}.

One of the primary problem areas in which such an algorithm becomes necessary is in learning from strategically generated or decision-dependent data, a classical problem in operations research (see, e.g., \cite{hellemo2018decision} and references therein).  This problem has garnered a lot of attention of late in the machine learning community under the name \say{performative prediction} \cite{Perdomo2020PerformativePrediction,Miller2021OutsideTheEchoChamber,brown2020performative} due to the growing recognition that learning algorithms are increasingly dealing with data from strategic agents. In such problems, assuming access to the response map of strategic agents is often too restrictive, and the introduction of agent's strategic responses into a convex loss function can often result in non-convex objectives. 

As an example of such a decision-dependent problem, consider a scenario in which a ride-sharing platform seeks to devise an adaptive pricing strategy which is responsive to changes in supply and demand. The platform observes the current supply and demand in the environment and adjusts the price to increase the supply of drivers (and potentially decrease the demand) as needed. Drivers, however, have the ability to adjust their availability, and can strategically create dips in supply to trigger price increases. Such gaming has been observed in real ride-share markets (see, e.g., \cite{hamilton2019uber,youn2019uber}) and results in negative externalities like higher prices for passengers. Importantly, in this situation, the platform does not observe precisely the decision making process of the drivers, only their strategically generated availability, and must learn to optimize through these agents' responses. This lack of precise knowledge regarding the data generation process, and the reactive nature of the data, motivate the use of game theoretic abstractions for the decision problem, as well as algorithms for finding solutions in the absence of full information.

Previous work analyzing this problem studies this phenomenon through the lens of risk minimization in which the data distribution is decision-dependent, and seeks out settings in which the decision maker can optimize the decision-dependent risk \cite{Miller2021OutsideTheEchoChamber}.  These works, however, do not account for model misspecification in their analysis. In particular, if the data generation model is incorrect, the performance of the optimal solution returned by their training methods may potentially degrade rapidly, something we explore in our experiments. 

We show that the decision-dependent learning (performative prediction) problem can be \emph{robustified} by taking a distributional robustness perspective on the original problem. Moreover, we show that, under mild assumptions, the distributionally robust decision-dependent learning problem can be transformed to a min-max problem and hence our zeroth-order random reshuffling algorithm can be applied. The gradient-free nature of our algorithm is important for applications where data is generated by strategic users that one must query; in these scenarios, the decision-maker is unlikely to access the best response map (data generation mechanism) of the strategic users, and hence will lack access to precise gradients. 



\paragraph{Contributions.} 
In this paper, we analyze the class of convex-concave min-max problems given by
\begin{align} \label{Eqn:Min-maxProblem}
    \min_{x \in \X} \max_{y\in \Y} \ L(x, y),
\end{align} 
where $\X \subset \R^{d_x}$, $\Y \subset \R^{d_y}$, and $L: \R^{d_x} \times \R^{d_y} \ra \R$ has the finite-sum structure given by $L := \frac{1}{n} \sum_{i=1}^n L_i$, where $L_1, \ldots, L_n: \R^{d_x} \times \R^{d_y} \ra \R$ denote $n$ individual loss functions. This formulation is ubiquitous in machine learning applications, where the overall loss objective is often the average of the loss function evaluated over each data point in a dataset. 

The contributions of this paper can be summarized as follows. 
\begin{enumerate}[label = \Roman*), align = left, leftmargin=*, topsep=-2pt, itemsep=-1pt]
    \item We propose an efficient zeroth-order random reshuffling-based OGDA algorithm for a convex-concave min-max optimization problem, \emph{without} assuming any other structure on the curvature of the min-max loss (e.g., strong convexity or strong concavity). We provide (to our knowledge) the \emph{first} non-asymptotic analysis of OGDA algorithm \emph{with} random reshuffling and zeroth-order gradient information. 
    
    \item As an important application, we formulate
    the \emph{Wasserstein distributionally robust learning with decision-dependent data} problem as a constrained finite-dimensional, smooth convex-concave min-max problem of the form \eqref{Eqn:Min-maxProblem}. In particular, we consider the setting of learning from strategically generated data, where the goal is to fit a generalized linear model, and where an ambiguity set is used to capture model misspecification regarding the data generation process. This setting encapsulates a distributionally robust version of the recently introduced problem of strategic classification \cite{Hardt2016StrategicClassification}. We show that this problem, under mild assumptions on data generation model and the ambiguity set, can be transformed into a convex-concave min-max problem to which our algorithm applies.
    
    \item We complement the theoretical contributions of this paper by presenting illustrative numerical examples.
    
\end{enumerate}

\section{Related Work}
\label{sec:RelatedWorks}




Our work draws upon the existing literature on zeroth-order methods for min-max optimization problems, decision-dependent learning (performative prediction), and  distributionally robust optimization.

\paragraph{Zero-Order Methods for Min-Max Optimization.}
\label{subsec: ZO Methods for Min-Max Optimization}

Zeroth-order methods provide a computationally efficient method for applications in which first-order or higher-order information is inaccessible or impractical to compute, e.g., when generating adversarial examples to test the robustness of black-box machine learning models \cite{Liu2019MinMaxWithoutGradients, Liu2020PrimerOnZOOptimization, Chen2017ZerothOrderOptimizationBasedBlackboxAttacks, Ilyas2018BlackboxAdversarialAttacks, Tu2019AutoZOOM, li2019nattack, Audet2017DerivativeFreeAndBlackboxOptimization}. 
Recently, Liu et al. and others \cite{Liu2019MinMaxWithoutGradients, Gao2018OnTheInformationAdaptiveVariantsOfTheADMM, Wang2020ZerothOrderAlgorithmsforNonconvexMinimaxProblems} provided the first non-asymptotic convergence bounds for zeroth-order algorithms, based on analysis methods for gradient-free methods in convex optimization \cite{Nesterov2017RandomGradientFreeMinimization}. However, these works assume that the min-max objective is either strongly concave in the maximizing variable \cite{Liu2019MinMaxWithoutGradients, Wang2020ZerothOrderAlgorithmsforNonconvexMinimaxProblems} or strongly convex \cite{Gidel2017FrankWolfeAlgorithmsforSaddlePointProblems} in the minimizing variable, an assumption that fails to hold in many applications \cite{Dong2018StrategicClassification, YuMazumdar2021FastDistributionallyRobustLearning}. In contrast, the zeroth-order algorithm presented in this work provides non-asymptotic guarantees under the less restrictive assumption that the objective function is convex-concave. In particular, we present the first (to our knowledge) zeroth-order variant of the Optimistic Gradient Ascent-Descent (OGDA) algorithm \cite{MokhtariOzdaglar2020ConvergenceRO, MokhtariOzdaglar2020AUnifiedAnalysis}. 

 In single-variable optimization problems, first-order stochastic gradient descent algorithms are empirically observed to converge faster when random reshuffling (RR, or sampling without replacement) is deployed, compared to sampling with replacement \cite{RechtRe2011ParallelSGA, Bottou2009CuriouslyFastConvergenceOfSomeSGDAlgorithms}. Although considerably more difficult to analyze theoretically, gradient-based RR methods have recently been shown to enjoy faster convergence when the underlying objective function is convex \cite{Shamir2016WithoutReplacementSF, JainNagaraj2019SGDWithoutReplacement, Mishchenko2021RandomReshufflingSimpleAnalysis, HaoChen2019RandomSB, Safran2019HowGoodIsSGDwithRandomReshuffling, GurbuzbalabanOzdaglar2021WhyRRBeatsSGD, Rajput2020ClosingConvergecneGapSGDWithoutReplacement}. Recently, these theoretical results have been extended to first-order methods for convex-concave min-max optimization problems \cite{YuMazumdar2021FastDistributionallyRobustLearning}; in this paper, we further extend these results to the case of zero-order algorithms. Specifically, we present the first (to our knowledge) non-asymptotic convergence rates for zeroth-order random reshuffling min-max optimization algorithms.
 
 
 
 
\paragraph{Distributionally Robust Optimization.}
\label{subsec: DRSL}

\textit{Distributionally Robust Optimization (DRO)} seeks to find solutions to optimization problems (e.g., supervised learning tasks) robust against changes in the data distribution between training and test time \cite{Madry2017TowardsDeepLearningModelsResistanttoAdversarialAttacks, YuMazumdar2021FastDistributionallyRobustLearning}. These distributional differences may arise due to imbalanced data, sample selection bias, or adversarial perturbations or deletions \cite{Quionero-Candela2009DatasetShiftInML, Madry2017TowardsDeepLearningModelsResistanttoAdversarialAttacks}, and are often modeled as min-max optimization problems, in which the classifier and an adversarial noise component are respectively modeled as the minimizer and maximizer of a common min-max loss objective \cite{Bagnell2005RobustSupervisedLearning, Bertsimas2010TheoryApplicationsOfRobustOptimization, Gabrel2014RecentAdvancesRobustOptimizationOverview, Gorissen2015PracticalGuideToRobustOptimization, RahimanMerotra2019DROReview}. In particular, the noise is assumed to generate the worst possible loss corresponding to a bounded training data distribution shift, with the bound given by either the $f$-divergence or Wasserstein distance. \cite{YuMazumdar2021FastDistributionallyRobustLearning, BenTal2013RobustSolutionsOptimizationProblems, NamkoongDuchi2016SGMforDRO, Hu2018DoesDRSLGiveRobustClassifiers, ShafieezadehAbadeh2015DistributionallyRL}. While this work considers adversarial noise in generated data, largely in a worst-case context, it has yet to capture strategically generated data wherein a data source generates data via a best response mapping. 

\paragraph{Strategic Classification and Performative Prediction.}
\label{subsec: Strategic Classification}
\textit{Strategic classification} \cite{Hardt2016StrategicClassification, Dong2018StrategicClassification, Khajehnejad2019OptimalDecisionMaking, SessaKamgarpour2020LearningToPlaySequentialGames} and \textit{performative prediction} \cite{Perdomo2020PerformativePrediction, MendlerDunner2020StochasticOptimization, Miller2021OutsideTheEchoChamber,  Drusvyatskiy2020StochasticOptimization}  concern supervised learning problems in which the training data distribution shifts in response to the deployed classifier or predictor more generally. This setting naturally arises in machine learning applications in which the selection of the deployed classifier either directly changes the training data (e.g., decisions based on credit scores, such as loan approvals, themselves change credit scores), or prompts the data source to artificially alter their attributes (e.g. withdrawals during bank runs spur worried clients to make more withdrawals) \cite{Perdomo2020PerformativePrediction, MendlerDunner2020StochasticOptimization, Miller2021OutsideTheEchoChamber}. Here, the learner accesses only perturbed features representing the strategic agents' best responses to a deployed classifier, and not the true underlying features \cite{Dong2018StrategicClassification}. This is a recently introduced formulation to machine learning; the results in this body of literature (to our knowledge) have not introduced the concept of robustness to model misspecification or the data generation process, in the same manner as we capture in this work.

\section{Preliminaries}
\label{sec: Preliminaries}

Recall that in this paper, we consider the class of convex-concave min-max problems given by:
\begin{align} \label{Eqn:Min-maxProblem2}
    \min_{x \in \X} \max_{y\in \Y} \ L(x, y),
\end{align} 
where $\X \subset \R^{d_x}$, $\Y \subset \R^{d_y}$, and $L := \frac{1}{n} \sum_{i=1}^n L_i$, where $L_1, \ldots, L_n: \R^{d_x} \times \R^{d_y} \ra \R$ denote $n$ individual loss functions. For convenience, we denote $d := d_x + d_y$. 

\begin{assumption} \label{Assumption:MainTheorem}
The following statements hold:
\begin{enumerate}[label=(\roman*),align=left,leftmargin=*,widest=ii, topsep=0pt, itemsep=0pt]
    \item\label{assm: compact} The sets $\X \subset \R^{d_x}$ and $\Y \subset \R^{d_y}$ are convex and compact.
    
    \item \label{assm: Lipschitz} The functions $L_1, \cdots, L_n: \R^d \ra \R$ are convex in $x \in \R^{d_x}$ for each $y \in \R^{d_y}$, concave in $y \in \R^{d_y}$ for each $x \in \R^{d_x}$, and $G$-Lipschitz and $\ell$-smooth in $(x, y) \in \R^d$ (which implies that $L: \R^d \ra \R$, by definition, also possesses the same properties).
\end{enumerate}
\end{assumption}

For ease of exposition, we denote $u := (x, y)$, $M_L := \sup_{u \in \X \times \Y} |L(u)|$, $D := \sup_{u, u' \in \X \times \Y} \| u - u' \|_2$, and define the operators $F, F_i: \R^d \ra \R^d$, for each $i \in [n]$, by:
\begin{align} \label{Eqn: F, Def}
    F(u) &:= \begin{bmatrix}
    \nabla_{x} L(u) \\
    - \nabla_{y} L(u)
    \end{bmatrix}, \hspace{1cm}
    F_i(u) := \begin{bmatrix}
    \nabla_{x} L_i(u) \\
    - \nabla_{y} L_i(u)
    \end{bmatrix}, \ \ \forall \ i \in [n].
\end{align}
Observe that under Assumption \ref{Assumption:MainTheorem}, $M_L, D < \infty$, and $F$ and each $F_i$ are monotone\footnote{A function \(F:\R^d\ra\R^d\) is called \emph{monotone} if \(\lara{ F(x)- F(y),x-y}\geq 0\) for all \(x,y\in\R^d\).}. Finally, we define the \textit{gap function $\Delta: \R^d \ra [0, \infty)$} associated with the loss $L$ by
\begin{align} \label{Eqn: Gap Function}
    \Delta(x, y) := L(x, y^\star) - L(x^\star, y) \geq 0,
\end{align}
where $u^\star := (x^\star, y^\star) \in \X \times \Y$ denotes any min-max saddle point of the overall loss $L(x, y)$, and $(x, y) \in \X \times \Y$ denotes any feasible point. This gap function allows us to measure the convergence rate of our proposed algorithm. To this end, we define the $\epsilon$-optimal saddle-point of \eqref{Eqn:Min-maxProblem2} as follows.

\begin{definition}[\textbf{$\epsilon$-optimal saddle point solution}]
A feasible point $(x, y) \in \X \times \Y$ is said to be an $\epsilon$-optimal saddle-point solution of \eqref{Eqn:Min-maxProblem2} if
\begin{align*}
    \Delta(x, y) = L(x, y^\star) - L(x^\star, y) \leq \epsilon,
\end{align*}
\end{definition}






\section{Algorithms and Performance}
\label{sec: Algorithms and Performance}

In this section we introduce a gradient-free version of the well-studied Optimistic Gradient Descent Ascent (OGDA) algorithm, and give finite time rates showing that it can efficiently find the saddle point in constrained convex-concave problems.
\subsection{Zero-Order Gradient Estimates}
\label{subsec: Zero-Order Gradient Estimates}

In our zero-order, random-reshuffling based variant of the OGDA algorithms, we use the one-shot randomized gradient estimator in \cite{Spall1997AOneMeasurementFormofSimultaneousPerturbationStochasticApproximation, Flaxman2005OnlineConvexOptimization, Gao2018OnTheInformationAdaptiveVariantsOfTheADMM, Liu2019MinMaxWithoutGradients}.
In particular, given the current iterate $u \in \R^d$ and a \textit{query radius} $\varepsilon > 0$, we sample a vector \(v\) uniformly from unit sphere \(\sphere^{d-1}\) (i.e. $v \sim \Unif(\sphere^{d-1})$), and define the zeroth-order estimator $\hat{F}(u; \varepsilon, v) \in \R^d$ of the min-max loss $L(u)$ to be:
\begin{align*}
    \hat{F}(u; \varepsilon, v) := \frac{d}{\varepsilon} L(u + \varepsilon v) v
\end{align*}
Properties of this zeroth-order estimator, derived in \cite{Bravo2018BanditLearning}, are reproduced as Proposition A.4 in Appendix A.1.

\subsection{Optimistic Gradient Descent Ascent with Random Reshuffling (OGDA-RR)}
\label{subsec: OGDA-RR}

In this subsection, we formulate our main algorithm, Optimistic Gradient Descent Ascent with Random Reshuffling (OGDA-RR). In each epoch $t \in \{0, 1, \cdots, T-1\}$, the algorithm generates a uniformly random permutation $\sigma^t := (\sigma_1^t, \cdots, \sigma_n^t)$ of $[n] := \{1, \cdots, n\}$ independently of any other randomness, and fixes a query radius $\epsilon^t > 0$ and search direction $v_i^t \in \R^d$. (Note: query radii only depends on epoch indices $t$, and not on sample indices). For each index $i \in [n]$, we compute the OGDA-RR update as follows:
\begin{align} \label{Eqn: OGDA-RR, Update}
    {\normalsize u_{i+1}^t = \proj_{\X \times \Y} \Big(u_{i}^t - \eta^t \hat{F}_{\sigma_i^t}(u_{i}^t; \varepsilon^t,v_{i}^t)-\eta^t \hat{F}_{\sigma_{i-1}^t}(u_{i}^t;\varepsilon^t,v_{i}^t) + \eta^t\hat{F}_{\sigma_{i-1}^t}(u_{i-1}^t;\varepsilon^t,v_{i-1}^t) \Big)},
\end{align}
After repeating this process for $T$ epochs, the algorithm returns the step-size-weighted average of the iterates, $\tilde u^T := \frac{1}{n \cdot \sum_{t=0}^{T-1} \eta^t} \sum_{t=0}^{T-1} \sum_{i=1}^n \eta^t u_i^t$. Roughly, the weighting described in Theorem \ref{Thm: OGDA Convergence} below optimally balances the bias and variance of the zero-order gradient estimator in Section \ref{subsec: Zero-Order Gradient Estimates}.




 




 
 

\begin{algorithm} 

\SetAlgoLined

 \textbf{Input}: stepsizes $\eta^t,\varepsilon^t$, data points \(\{(x_i,y_i)\}_{i=1}^{n}\sim\mathcal{D},u_0^{(0)}\), time horizon duration $T$;
 
 \For{$t=0,1, \cdots, T-1$}{
$\sigma^t = (\sigma_1^t, \cdots, \sigma_n^t) \gets$ a random permutation of set \([n]\)\;

\For{$i=0,\ldots, n-1$}{
Sample $v_i^t \sim  \Unif(\sphere^{d-1})$

$u_{i+1}^t \gets \eqref{Eqn: OGDA-RR, Update}$


}

$u_{0}^{(t+1)} \gets u_{n}^t$

$u_{-1}^{(t+1)} \gets u_{n-1}^t$

 }
 
\textbf{Output:} $\tilde u^T := \frac{1}{n \cdot \sum_{t=0}^{T-1} \eta^t} \sum_{t=0}^{T-1} \sum_{i=1}^n \eta^t u_i^t$.
 
\caption{OGDA-RR Algorithm}
\label{Alg: OGDA-RR}
\end{algorithm}

\begin{theorem} \label{Thm: OGDA Convergence}
Let $L(u)$ denote the objective function in the constrained min-max optimization problem given by \eqref{Eqn:Min-maxProblem}, and let $u^\star=(\x^\star, \y^\star) \in \X \times \Y$ denote any saddle point of $L(u)$. Fix $\epsilon > 0$. Suppose Assumption \ref{Assumption:MainTheorem} holds, and the number of epochs $T$, step sizes sequence $\{\eta^t\}_{t=0}^{T-1}$, and query radii sequence $\{\varepsilon^t\}_{t=0}^{T-1}$ satisfy:
\begin{align*}
    \eta^t &:= \eta^0 \cdot (t+1)^{-3/4 + \chi}, \hspace{1cm} \forall \hspace{0.5mm} t \in \{0, 1, \cdots, T-1\}, \\
    \varepsilon^t &:= \varepsilon^0 \cdot (t+1)^{-1/4}, \hspace{1cm} \forall \hspace{0.5mm} t \in \{0, 1, \cdots, T-1\}, \\
    T &> \frac{
    1}{\varepsilon^4} \Bigg( \frac{3}{16n} D + \frac{5}{4n} \cdot C \cdot \max\left\{ \varepsilon^0, \eta^0, \eta^0 \varepsilon^0, \frac{\eta^0}{\varepsilon^0},  \frac{\eta^0}{(\varepsilon^0)^2} \right\} \Big( 1 + \frac{1}{\chi} \Big) \Bigg)^{\frac{4}{1-4\chi}},
\end{align*}
for some initial step size $\eta^0 \in \Big(0, \frac{1}{2\ell} \Big)$, initial query radius $\varepsilon^0 > 0$, parameter $\chi \in (0, 1/4)$, and constant:
\begin{align*}
    C &:= \max\big\{3nd D \ell, 18nd D G\ell, 54 ndG^2 + 18 ndD \ell M_L, 90nd GM_L, 36 ndM_L^2, \\
    &\hspace{2cm} 6 dG n^2 + 14 Gn^2 + 4nG, 6dM_L n, 3d D G\ell, 3d D \ell M_L \big\} > 0.
\end{align*}
Then the iterates $\{u_i^t\}$ generated by the OGDA-RR Algorithm (Alg. \ref{Alg: OGDA-RR}) satisfy:
\begin{align*}
    \E\big[\Delta(\tilde u^T) \big] < \epsilon.
\end{align*}
\end{theorem}

\begin{remark}
Note that our OGDA-RR algorithm is more computationally efficient than Alg. 2 in \cite{YuMazumdar2021FastDistributionallyRobustLearning}, \emph{even if} one replaces the gradient estimates with true gradient values. This is because Alg. 2 in \cite{YuMazumdar2021FastDistributionallyRobustLearning} requires \(M\sim \mc{O}(\log(n))\) inner loop iterations to approximate a proximal point update. Here, we avoid this restriction by exploiting the recent perspective that the OGDA update is a perturbed proximal point update \cite{MokhtariOzdaglar2020AUnifiedAnalysis,MokhtariOzdaglar2020ConvergenceRO}. For more details, see Appendix A.2 for the proof of Theorem \ref{Thm: OGDA Convergence}. 
\end{remark}




\section{Applications to Decision-Dependent DRO}
\label{sec: Applications (or Experiments)}

In this section we discuss a novel convex-concave min-max reformulation of a class of decision-dependent distributional robust risk minimization problems, which reflects the need for learning classifiers that are simultaneously robust to strategic data sources and adversarial model-specification. In particular, we present a distributionally robust formulation of strategic classification \cite{Dong2018StrategicClassification} with generalized linear loss, a semi-infinite optimization problem that can be reformulated to a finite-dimensional convex-concave min-max problem.


Strategic classification is an emerging paradigm in machine learning which attempts to ``close the loop"--- i.e., account for data (user) reaction at training time---while designing classifiers to be deployed in strategic environments in the real world, where deploying \emph{na\"ive} classifier (designed ignoring the distribution shift) can be catastrophic.
Modeling the exact behavior of such strategic interactions is very complex, since the decision-maker (learner) does not have access to the strategic users' preferences and hence lacks access to their best response function. To overcome this difficulty, we use a natural model for these strategic behaviors that has been exploited in Dong et. al.(2018), and then impose robustness conditions (in the form of an ambiguity set on the decision-dependent  data distribution) to capture model misspecification.
To facilitate the discussion, we provide a primer on decision-dependent DRO in the next subsection. 

\subsection{Primer on decision-dependent distributionally robust optimization}

Consider a \emph{generalized linear problem}, where the goal is to estimate the parameter \(\param\in \Theta\), which is assumed to be a compact set, by solving the following convex optimization program:
\[
\inf_{\param\in\Theta} \avg_{\mc{D}} \ls{ \phi\lr{\lara{\bx,\param}} - \by \lara{\bx,\param} }
\]
where \(\phi: \R\ra \R\) is a smooth convex function and the tuple \((\bx,\by)\in \R^d\times\{-1,+1\}\) is sampled from an unknown distribution \(\mc{D}\), often approximated by the empirical distribution of a set of observed data.
The generalized linear model encompasses a wide range of machine learning formulations \cite{Mccullagh2019GeneralizedLinearModels}.

A distributionally robust generalized linear problem, on the other hand, minimizes the worst case expectation over an uncertainty set \(\pSet\) in the space of probability measures. This setup can be envisioned as a game between a learning algorithm and an adversary. Based on parameters chosen by the learning algorithm, the adversary then picks a probability measure from the uncertainty set which maximizes the risk for that choice of parameter:
\[
\inf_{\param} \sup_{\dist\in\pSet} \avg_{\dist}\ls{ \phi\lr{\lara{\bar{x},\param}} - \bar{y} \lara{\bar{x},\param} },
\]
where \((\bar{x},\bar{y})\sim\mb{P}\in \mc{P}\).
Typically \(\pSet\) is chosen as a Wasserstein ball around the empirical distribution \(\td{\mc{D}}_\numData\) of a set of $\numData$ observed data points, $\{(\xData_i,\yData_i) \in \R^d\times \{-1,1\}\}_{i=1}^\numData$, sampled independently from the data distribution \(\mc{D}\). Then, for any 
\(\delta>0\) the uncertainty set \(\pSet\) is given by \(
\mb{B}_{\delta}(\td{\mc{D}}_{\numData}) = \{ \dist: 
\mc{W}(\dist,\td{\mc{D}}_\numData) \leq \delta \}\).  

A critique of the above problem formulation is that the underlying data distribution \(\mc{D}\) is considered fixed, while in many strategic settings underlying data distribution will depend on the classifier parameter \(\param\). \emph{Decision-dependent supervised learning} aims to tackle such distribution shifts. When specialized to the generalized linear model, the problem formulation becomes:
\[
\inf_{\param}\avg_{\mc{D}(\param)}\ls{ \phi\lr{\lara{\bar{x},\param}} - \bar{y} \lara{\bar{x},\param} }, 
\]
where \((\bar{x},\bar{y})\sim\mc{D}(\theta)\).
In this work, we take a step forward and work with the \emph{distributionally robust decision-dependent generalized linear model}, defined as:
\begin{align} \label{Eqn: WDRSC Problem}
\inf_{\param}\sup_{\dist\in\pSet(\param)}\avg_{\dist}\ls{ \phi\lr{\lara{\bar{x},\param}} - \bar{y} \lara{\bar{x},\param} }, 
\end{align}
where \((\bar{x},\bar{y})\sim \mb{P}\in \mc{P}(\theta)\) and
\(\pSet(\param)=\mb{B}_{\delta}(\td{\mc{D}}_{\numData}(\param))\). Here, the dependence of $\mathbb{P}$ on the choice of classifier $\theta$ is captured by its inclusion in \(\pSet(\param)=\mb{B}_{\delta}(\td{\mc{D}}_{\numData}(\param))\). To describe decision-dependent distribution shifts \(\td{\mc{D}}_n(\theta)\), we restrict our focus to the setting of strategic classification. The following subsection formalizes our setting.

\subsection{Model for strategic response}\label{ssec: ModelOfResponse}
 Below, we denote the data points sampled from \emph{true distribution} by $(\xData_i,\yData_i)\sim\mc{D}$ where $\mc{D}$ is a unknown, underlying distribution. For ease of presentation, we associate each data point index $i$ with an agent. For each agent \(i\in[n]\), let \( u_i(x;\theta,\xData_i,\yData_i) \in \R\) denote its utility function that a strategic agent seeks to maximize.  In other words, when a classifier parametrized by \(\theta\in\R^d\) is deployed, the agent \(i\in[n]\) responds by reporting  \(b_i(\theta,\xData_i,\yData_i)\), defined as:
 \[
 b_i(\theta,\xData_i,\yData_i) \in \arg\max_x u_i(x;\theta,\xData_i,\yData_i).
 \]
Note that we allow different agent to have different utility function. 

We now impose the following assumptions on the utility functions; these are crucial for ensuring guaranteed convergence of our proposed algorithms.

\begin{assumption}\label{assm:2}
For each agent \(i\in[n]\), define \(u_i(x;\theta,\xData_i,\yData_i)\defas  \frac{1-\yData_i}{2}\lara{ x,\theta}-g_i(x-\xData_i)\),  where $g_i:\mb{R}^d\to \mb{R}$ satisfies:
    \begin{enumerate}[label=(\roman*),align=left,leftmargin=*,widest=iii,itemsep=-1mm]
        \item \label{assm:21}$g_i(x)>0$ for all $x\neq 0$;
        \item \label{assm:22} $g_i$ is convex on $\mb{R}^d$;
        \item \label{assm:23}$g_i$ is positive homogeneous\footnote{A function \(f:\R^d\ra\R\) is \emph{positive homogenous of degree r} if for any scalar \(\alpha>0\) and \(x\in\R^d\) we have \(f(\alpha x) = \alpha ^ r f(x)\)} of degree $\homDegree>1$;
        \item \label{assm:24} Its convex conjugate \(g_i^\ast(\theta)\defas \sup_{x\in\R^d}\lara{ x,\theta} - g_i(x)\) is \(G_i\)-Lipschitz and \(\bar{G}_i\)-smooth on \(\Theta\).
    \end{enumerate}
\end{assumption}

As is pointed out in Dong et. al. (2018), a large class of functions \(g(\cdot)\) satisfy the requirements posited in Assumption \ref{assm:2}. For example, for any arbitrary norm and any \(\homDegree>1\) the function \(g(x) = \frac{1}{\homDegree}\Vert x \Vert^\homDegree\)  is a candidate.  
Note that these assumptions are not very restrictive and capture a large variety of practical scenarios \cite{Dong2018StrategicClassification}. A natural consequence of the above modeling paradigm is that \(b_i(\theta,\xData_i,+1) = \xData_i\). To wit, the agents act strategically only if their true label is \(-1\). This is a reasonable setting for many real world applications \cite{Dong2018StrategicClassification}. We now present a technical lemma which will be helpful in subsequent presentation. 
\begin{lemma}[Dong et. al. (2018)] \label{lem: ConvexityBR}
  Under Assumption \ref{assm:2}, for each agent \(i\in[n]\), the set of best responses \(\arg\max_x u_i(x;\theta,\xData_i,\yData_i)\) is finite and bounded. The function \(\theta\mapsto \lara{ b_i(\theta,\xData_i,\yData_i),\theta}\) is convex. To wit, for any \(i\in[n]\): \(\lara{ b_i(\theta,\xData_i,\yData_i), \theta} = \lara{ \xData_i,\theta} + \frac{1-\yData_i}{2}q g_i^*(\theta)\) where \(\frac{1}{\homDegree}+\frac{1}{q} = 1\)
\end{lemma}

Against the preceding backdrop, we now present the convex-concave min-max reformulation of the Wasserstein Distributionally Robust Strategic Classification (WDRSC) problem.

\subsection{Reformulation of the WDRSC Problem}
\label{subsec: Reformulation of WDRSC}


The WDRSC problem formulation contains two main components---the \emph{strategic component} that accounts for a distribution shift \(\mc{D}(\theta)\) in response to the choice of classifier \(\theta\), and the \emph{adversarial component} that accounts for the uncertainty set \(\mc{P}(\theta)\). As per the modeling assumptions described in Section \ref{ssec: ModelOfResponse}, we have \((\xData_i,\yData_i) \sim \mc{D}\) and \((b_i(\theta,\xData_i,\yData_i),\yData_i)\sim \mc{D}(\theta)\) for all \(i\in[n]\). We now impose certain restrictions on the adversarial component that would enable us to reformulate the WDRSC problem as a convex-concave min-max optimization problem. Crudely speaking, we allow adversarial modifications on \emph{features} for all data points, but adversarial modifications on \emph{labels} only when the true label is \(+1\).

{For the distributionally robust strategic classification problem, we consider a specific form of uncertainty set \(\mc{P}(\theta)\) that allows us to reformulate the infinite-dimensional optimization problem as a finite-dimensional convex-concave min-max problem.  As described above, in our formulation, the features of a given data point $i$ can be perturbed strategically if $\Tilde{y}_i = -1$, but not if $\Tilde{y}_i = +1$. On top of the strategic perturbations we also consider the adversarial perturbations to the data points. Specifically, we also assume that the adversary can perturb both the features and label of a data point $i$ if $\Tilde{y}_i = 1$, but can only perturb the features and not the label if $\Tilde{y}_i = -1$. A rigorous exposition of this restriction is deferred to Appendix B.1. Under these assumptions, we now present a convex-concave min-max reformulation of the WDRSC problem.
}

\begin{theorem} \label{thm: WDRSC-convex-concave}
Let the strategic behavior of the agents be governed in accordance with Assumption \ref{assm:2}. Suppose $\phi$ is convex and $\beta$-smooth. In addition, suppose \(\R\ni x\mapsto \phi(x)+x \in\R\) is non-decreasing. Then the WDRSC problem \eqref{Eqn: WDRSC Problem} can be reformulated into the following convex-concave min-max problem:
\begin{align} \label{eq:reformpp-new}
    &\min_{(\theta,\alpha)} \max_{\gamma\in \mb{R}^n} \Bigg\{\alpha (\delta-\kappa) +\frac{1}{n}\sum_{i}\frac{1+\yData_i}{2}\lr{\phi\lr{\lara{b_i(\theta),\theta}}} +\gamma_i \lr{\lara{b_i(\theta),\theta}-\alpha\kappa} \\ \nonumber
    &\hspace{1cm} +\frac{1}{n}\sum_i\frac{1-\yData_i}{2}\lr{ \phi(\lara{b_i(\theta),\theta})+\lara{b_i(\theta),\theta}} \Bigg\}\\ \nonumber
    &\text{s.t.} \Vert\theta\Vert\leq \alpha/(\beta+1), \ \Vert\gamma\Vert_\infty\leq 1
\end{align}
where for any \(i\in[n]\), we have concisely written \(b_i(\theta,\xData_i,\yData_i)\) as \(b_i(\theta)\). 
\end{theorem}
The proof of Theorem \ref{thm: WDRSC-convex-concave} is presented in Appendix B.2. 

\begin{remark}
The non-decreasing assumption on the map \(\R\ni x\mapsto \phi(x)+x \in\R\) is not overly restrictive; in fact, it is satisfied by the logistic regression model in supervised learning (see Appendix C). 
\end{remark}


{
\begin{remark}
Note that we can convert the smooth convex-concave minmax problem \eqref{eq:reformpp-new} into a non-smooth convex minimization problem by explictly taking maximization over \(\gamma\).  But we refrain from doing as it has been observed \cite{YuMazumdar2021FastDistributionallyRobustLearning} that solving the smooth minimax optimization problem is faster than solving the non-smooth problem. In fact, we have presented an experimental study in Appendix C which corroborates this observation. 
\end{remark}
}
Throughout the rest of this paper, we denote the min-max objective in \eqref{eq:reformpp-new} by \(L(\alpha,\theta,\gamma)\).


\section{Empirical Results}\label{sec: ExpResults}
{In this section we deploy zeroth-order OGDA algorithm with random reshuffling to solve the convex concave reformulation of WDRSC as presented in \eqref{eq:reformpp-new}. We point out that in order to solve \eqref{eq:reformpp-new}, the zeroth-order method should only be applied to estimate the gradient with respect to \(\theta\). This is because the gradient with respect to other variables, namely \((\alpha,\gamma)\), can be exactly computed. Specifically, to compute derivative with respect to \(\theta\) the designer must know the best response function which is often not available and it can only be queried. 

We now present some illustrations of the empirical performance of our proposed algorithm, as well as empirical justification for solving the WDRSC problem over existing prior approaches to strategic classification. 
}

\subsection{Experimental Setup}

Our first set of empirical results uses synthetic data to illustrate the effectiveness of our algorithms. The datasets used in this section are constructed as follows: the ground truth classifier \(\theta^\star\) and features \(\xData_i\) are sampled as \(\theta^\star \sim \mc{N}(0,I_d)\) and \(\xData_i \sim \text{ i.i.d. } \mc{N}(0,I_d)\), for each \(i\in[n]\), while the ground truth labels \(\yData_i \) are given by \(\yData_i  = \textsf{sign}(\lara{\xData_i,\theta^\star}+z_i)\) for each \(i\in[n]\), where \(z_i \sim \text{ i.i.d. } \mc{N}(0,0.1\cdot I_d)\). We use \(n\in \{500,1000\}\) with \(d=10\). The first five of the $d=10$ features are chosen to be strategic. In all experiments, we take \(\kappa=0.5\) and \(\delta=0.4\). 
Each strategic agent  \(i\in[n]\) has a utility function given by:
\begin{align} \label{Eqn: Utility Function}
    u_i(x;\theta,\xData_i,\yData_i,\zeta_i) = \frac{1-\yData_i}{2}\lara{x,\theta}-\frac{1}{2\zeta_i}\Vert x-\xData_i\Vert^2,
\end{align}
where \(\zeta_i\) denote the perturbation ``power" of agent \(i\). 
For simplicity, we assume all agents are homogeneous, in the sense that \(\zeta_i=\zeta > 0\) for all \(i\in[n]\); in practice, one need not impose this assumption. Given this utility function, the best response of agents takes the form:
\begin{align} \label{Eqn: Best Response}
    b_i(\theta,\xData_i,\yData_i;\zeta) = \begin{cases}
    \xData_i & \text{if} \ \yData_i=+1,\\
    \xData_i+\zeta\theta & \text{if} \ \yData_i = -1
    \end{cases}
\end{align}
where, in our simulations, we fix \(\zeta = 0.05\). We reemphasize that our algorithm does not use the value of \(\zeta\) in any of its computations.
For purposes of illustration, we focus on the performance of the following algorithms: 
\begin{enumerate}[label=(\text{A}-\Roman*),align=left,leftmargin=*,widest=iiii,itemsep=-2mm]
    \item \label{OGDARR} Zeroth-order optimistic-GDA \emph{with} random reshuffling (see Algorithm \ref{Alg: OGDA-RR}),
    \item \label{OGDAS} Zeroth-order optimistic-GDA \emph{without} random reshuffling (see Appendix C),
    \item \label{SGDARR} Zeroth-order stochastic-GDA \emph{with} random reshuffling (see Appendix C), 
    \item \label{SGDAS} Zeroth-order stochastic-GDA \emph{without} random reshuffling (see Appendix C).
\end{enumerate}

and we evaluate the proposed algorithms and model formulation on two criteria:
\begin{enumerate}[label=(\roman*),align=left,leftmargin=*]
    \item \textit{Suboptimality}: To measure suboptimality, we use the \emph{gap function} \(\Delta(\alpha,\theta,\gamma) = L(\alpha,\theta,\gamma^\star) - L(\alpha^\star,\theta^\star,\gamma) \) (Def. \ref{Eqn: Gap Function}) where \((\alpha^\star,\theta^\star,\gamma^\star)\) is a solution of the min-max reformulation \eqref{eq:reformpp-new} of the WDRSC problem. If the objective $L(\cdot)$ is convex-concave, $\Delta(\cdot)$ is non-negative, and equals zero at (and only at) saddle points. 
    
    \item \textit{Accuracy}: Given a    data set \(\{(\xData_i,\yData_i)\}_{i\in[n]}\), the accuracy of a classifier \(\theta\) is measured as \( \frac{1}{n}\sum_{i\in[n]}\yData_i\lara{ b_i(\theta,\xData_i,\yData_i;\zeta),\theta }\). Under this criterion we compare the accuracy under different perturbations  for different classifiers $\theta$;
\end{enumerate}
To compute suboptimality, we first compute a true min-max saddle point \((\alpha^ \star,\theta^\star,\gamma^\star)\) via a first order gradient based algorithm (namely, GDA). 
All experiments were run using Python 3.7 on a standard MacBook Pro laptop (2.6 GHz Intel Core i7 and 16 GB of RAM).

\subsection{Results}
Simulation results presented in Figure \eqref{fig:sub1}-\eqref{fig:sub2} show that our proposed algorithm (i.e. \ref{OGDARR}) outperforms algorithms without reshuffling (i.e. \ref{OGDAS} and \ref{SGDAS}). However, its performance resembles that of zeroth-order stochastic-GDA with random reshuffling. More experimental studies need to be conducted to more  conclusively determine whether \ref{OGDARR} outperforms \ref{SGDARR}, or vice versa. In fact , there has been no theoretical investigations even for the \emph{first order} stochastic-GDA algorithm with random reshuffling; this is an interesting future direction to explore. 

In Figure \ref{fig: SyntheticData}, we also compare the robustness of the classifier obtained by using Algorithm \ref{OGDARR} with that obtained from prior work on solving probems of strategic classification trained with \(\zeta=0.05\) (referred as \emph{LogReg SC} in Figure \ref{fig: SyntheticData}). As expected, due to the formulation, the performance of the classifier obtained via \ref{OGDARR} degrades gracefully even when subject to large perturbations, while the performance of existing approaches to strategic classification degrades rapidly. 
Further numerical results on synthetically generated and real world datasets are given in Appendix C.

 
\begin{figure}
\centering
\begin{subfigure}{\linewidth}
  \centering
  \includegraphics[width=0.95\linewidth]{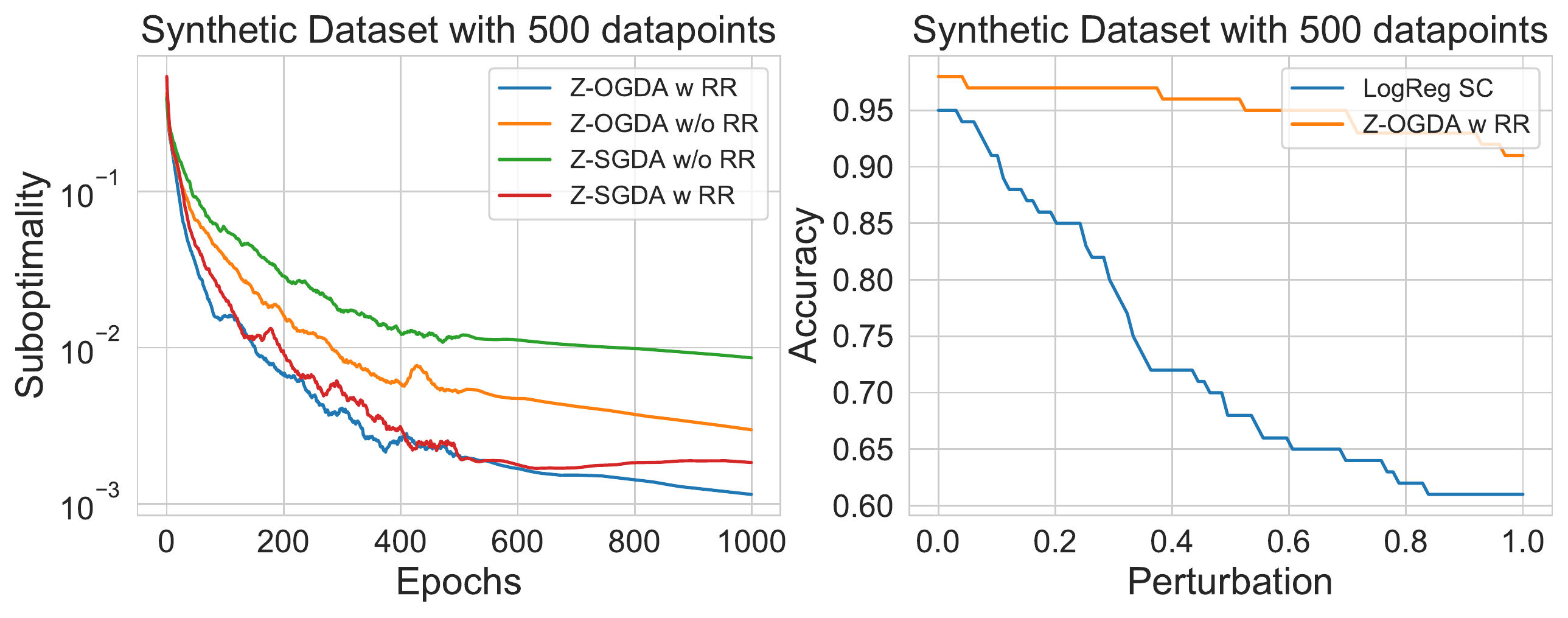}
  \caption{\(n=500\)}
  \label{fig:sub1}
\end{subfigure}\\
\begin{subfigure}{\linewidth}
  \centering
  \includegraphics[width=0.95\linewidth]{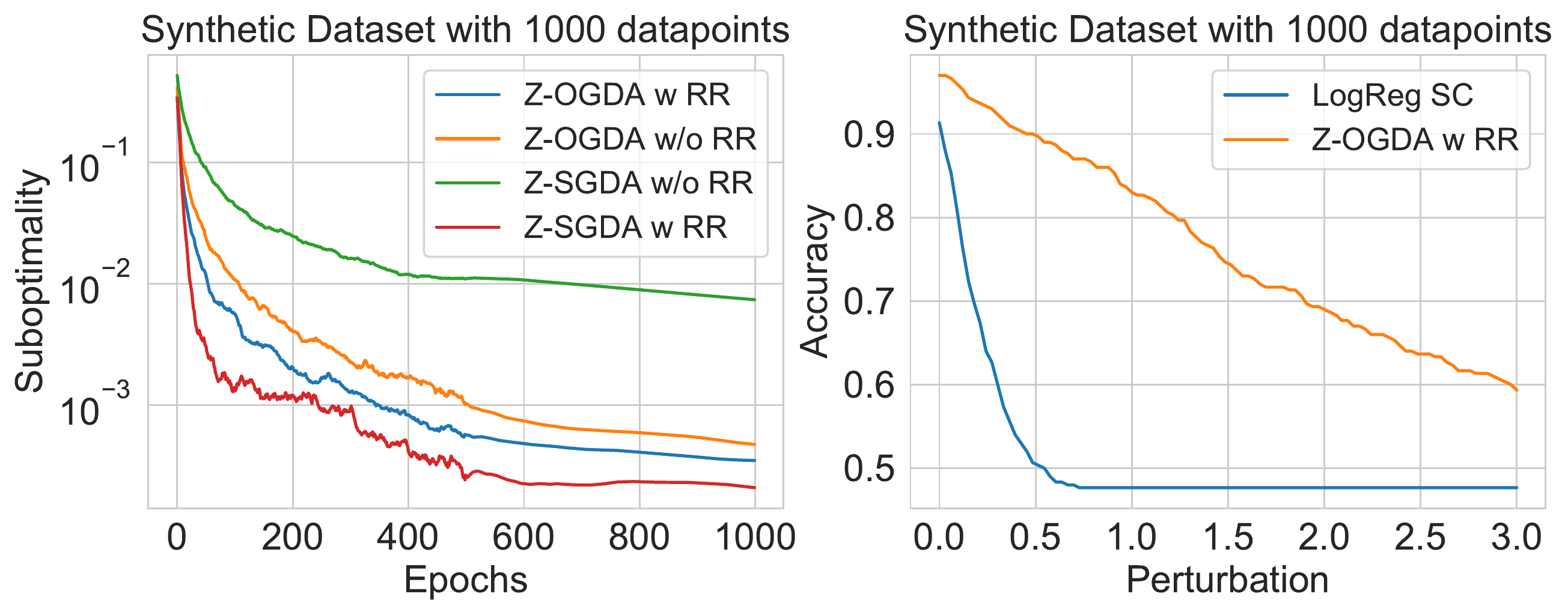}
  \caption{\(n=1000\)}
  \label{fig:sub2}
\end{subfigure}
\caption{Experimental results for a synthetic dataset with \(n=500\) and \(n=1000\). (Left panes of \eqref{fig:sub1}, \eqref{fig:sub2})) Suboptimality iterates generated by the four algorithms \ref{OGDARR}, \ref{OGDAS}, \ref{SGDARR}, \ref{SGDAS}, respectively denoted as \emph{Z-OGDA w RR}, \emph{Z-OGDA w/o RR}, \emph{Z-SGDA w RR}, \emph{Z-SGDA w/o RR}. (Right panes of \eqref{fig:sub1}, \eqref{fig:sub2})) Comparison between decay in accuracy of strategic classification with logistic regression (trained with \(\zeta=0.05\)) and Alg. \ref{OGDARR} with change in perturbation.}
\label{fig: SyntheticData}
\end{figure}


\section{Conclusion}
\label{sec: Conclusion}


This paper presents the first (to our knowledge) non-asymptotic convergence rates for a gradient-free stochastic min-max optimization algorithm with random reshuffling. Our theoretical results, established for smooth convex-concave min-max objectives, do not require any additional, restrictive structural assumptions to hold. As a concrete application, we reformulate a distributionally robust strategic classification problem as a convex-concave min-max optimization problem that can be iteratively solved using our method. Empirical results on synthetic and real datasets demonstrate the efficiency and effectiveness of our algorithm, as well as its robustness against adversarial distributional shifts and strategic behavior of the data sources. Immediate directions for future work include establishing convergence results for the random-reshuffling based Stochastic Gradient Descent Ascent (SGDA-RR) algorithm, as well as performing more extensive experimental studies to better understand the empirical performance of our algorithm.



\section*{Acknowledgements}
Research supported by NSF under grant DMS 2013985 ``THEORINet: Transferable, Hierarchical, Expressive, Optimal, Robust and Interpretable Networks''. Chinmay Maheshwari, Chih-Yuan Chiu and S. Shankar Sastry were also supported in part by U.S. Office of Naval Research MURI grant N00014-16-1- 2710. 

\bibliographystyle{alpha}
\bibliography{refs}

\newpage

\appendix
\section{Results for the Proof of Theorem 4.1}
\label{sec: Results for the Proof of Theorem 4.1}

\subsection{Lemmas for Theorem \ref{Thm: OGDA Convergence}}
\label{subsec: App, Lemmas for OGDA Theorem}

First, we list some fundamental facts regarding projections onto convex, compact subsets of an Euclidean space. Below, for any fixed convex, compact subset $\Omega \subset \R^d$, we denote the projection operator onto $\Omega$ by $\proj_\Omega(x) := \text{argmin}_{z \in \Omega} \Vert x - z \Vert_2$ for each $x \in \R^d$. Note that $\proj_\Omega(x)$ is well-defined (i.e., exists and is unique) for each $x \in \R^d$, if $\Omega \subset \R^d$ were convex and compact.

We begin by summarizing some fundamental properties of the projection operator $\proj_\Omega(\cdot)$.

\begin{proposition} \label{Prop: App, Projection, 1}
Let $\Omega \subset \R^d$ be compact and convex, and fix $x, y \in \R^d$ arbitrarily. Then:
\begin{align*}
    \big\Vert \proj_\Omega(x) - \proj_\Omega(y) \big\Vert_2^2 &\leq \big( \proj_\Omega(x) - \proj_\Omega(y) \big)^\top (x - y), \\
    \Vert \proj_\Omega(x) - \proj_\Omega(y) \Vert_2 &\leq \Vert x - y \Vert_2.
\end{align*}
\end{proposition}

\begin{proof}
From \cite{Nesterov2014ConvexOptimizationLecturesBasicCourse}, Lemma 3.1.4 (see also \cite{RechtWright2021OptimizationforDataAnalysis}, Lemma 7.4), we have:
\begin{align*}
    &\big( \proj_\Omega(x) - \proj_\Omega(y) \big)^\top \big(x - \proj_\Omega(x) \big) \geq 0, \\
    &\big( \proj_\Omega(y) - \proj_\Omega(x) \big)^\top \big(y - \proj_\Omega(y) \big) \geq 0.
\end{align*}
Adding the two expressions and rearranging terms, we obtain:
\begin{align*}
    &\big( \proj_\Omega(x) - \proj_\Omega(y) \big)^\top \big( (x-y) - (\proj_\Omega(x) - \proj_\Omega(y)) \big) \geq 0, \\
    \Rightarrow \hspace{0.5mm} &\Vert \proj_\Omega(x) - \proj_\Omega(y) \Vert_2^2 \leq \big( \proj_\Omega(x) - \proj_\Omega(y) \big)^\top (x - y),
\end{align*}
as given in the first claim. The Cauchy Schwarz inequality then implies:
\begin{align*}
    \Vert \proj_\Omega(x) - \proj_\Omega(y) \Vert_2^2 &\leq \big( \proj_\Omega(x) - \proj_\Omega(y) \big)^\top (x - y) \\
    &\leq \Vert \proj_\Omega(x) - \proj_\Omega(y) \Vert_2 \cdot \Vert x - y \Vert_2.
\end{align*}
If $\proj_\Omega(x) = \proj_\Omega(y)$, then the second claim becomes $0 \leq \Vert x - y \Vert_2$, which is clearly true. Otherwise, dividing both sides above by $\Vert \proj_\Omega(x) - \proj_\Omega(y) \Vert_2$ gives the second claim.
\end{proof}




\begin{lemma} \label{Lemma: Proj, applied to OGDA Thm}
Let $\Omega \subset \R^d$ be a compact, convex subset of $\R^d$, and consider the update \(z_{k+1} = \proj_\Omega( z_{k} -\eta F(z_{k+1}) + \gamma_k)\), where  $z_k, z_{k+1}, \gamma_k \in \R^d$. Then, for each $z \in \Omega$: 
\begin{align*}
    &\lara{F(z_{k+1}), z_{k+1}-z} \\ \leq \hspace{0.5mm} &\frac{1}{2\eta}\|z_k-z\|^2-\frac{1}{2\eta}\|z_{k+1}-z\|^2 -\frac{1}{2\eta}\|z_{k+1}-z_k\|^2 + \frac{1}{\eta} \lara{\gamma_k,z_{k+1}-z }.
\end{align*}    
\end{lemma}

\begin{proof}
Note that:
\begin{align*}
    \|z_{k+1}-z\|^2 &= \|z_{k+1}-z_k+z_k-z\|^2 \\ 
    &= \|z_{k+1}-z_k\|^2 + \|z_k-z\|^2 + 2 \lara{ z_{k+1}-z_k,z_k-z} \\ 
    &= \|z_{k+1}-z_k\|^2 + \|z_k-z\|^2 +  2 \lara{ z_{k+1}-z_k,z_k-z_{k+1}+z_{k+1} -z}  \\ 
    &= \|z_k-z\|^2-\|z_{k+1}-z_k\|^2 +  2 \lara{ z_{k+1}-z_k,z_{k+1} -z} 
\end{align*}
By definition of $z_{k+1}$, and optimality conditions for the projection operator:
\begin{align*}
    &\lara{z_{k+1}-z,z_{k+1}-z_k+\eta F(z_{k+1}) - \gamma_k}\leq 0, \\
    \Ra \hspace{0.5mm} &\lara{ z_{k+1}-z_k,z_{k+1} -z} \leq \lara{\gamma_k, z_{k+1}-z} - \eta \cdot \lara{F(z_{k+1}), z_{k+1}-z}.
\end{align*}
Substituting back, we obtain:
\begin{align*}
    \|z_{k+1}-z\|^2 &= \|z_k-z\|^2-\|z_{k+1}-z_k\|^2 +  2 \lara{ z_{k+1}-z_k,z_{k+1} -z} \\
    &\leq \|z_k-z\|^2 - \|z_{k+1}-z_k\|^2 + 2 \lara{\gamma_k, z_{k+1}-z} - 2 \eta \cdot \lara{F(z_{k+1}), z_{k+1}-z}.
\end{align*}
Rearranging and dividing by $\eta$ gives the claim in the lemma.
\end{proof}

Next, we state the properties of the mean and variance of the zeroth-order gradient estimator defined in Section \ref{subsec: Zeroth-Order Gradient Estimates} (\cite{Bravo2018BanditLearning}, Lemma C.1). Below, we define the $\queryRadius$-smoothed loss function $L^\queryRadius: \R^d \ra \R$ by $L^\queryRadius(u) := \E_{\overline{v} \sim \Unif(B^d)}[L(u + \queryRadius \overline{v})]$, where $\sphere^{d-1}$ denotes the ($d-1$)-dimensional unit sphere in $\R^d$, $B^d$ denotes the $d$-dimensional unit open ball in $\R^d$, and $\Unif(\cdot)$ denotes the continuous uniform distribution over a set. Similarly, we define $L_i^\queryRadius: \R^d \ra \R$ by $L_i^\queryRadius(u) := \E_{\overline{v} \sim \Unif(B^d)}[L_i(u + \queryRadius \overline{v})]$, for each $i \in [n] := \{1, \cdots, n\}$. We further define $\queryRadius \cdot \sphere^{d-1} := \{\queryRadius v: v \in \sphere^{d-1} \}$ and $\queryRadius \cdot B^d := \{\queryRadius \overline{v}: \overline{v} \in B^d \}$. Finally, we use $\vol_d(\cdot)$ to denote the volume of a set in $d$ dimensions.

\begin{proposition} \label{Prop: App, ZO grad mean, variance}
Let $\hat{F}(u; \queryRadius, v) = \frac{d}{\queryRadius} \cdot L(u + \queryRadius v) v$ and \(F(u)=\nabla L(u)\). Then the following holds: 
\begin{align} 
\label{Eqn: Prop, ZO grad mean, equality}
    \E_{v \sim \Unif(\sphere^{d-1})} \big[\hat{F}(u; \queryRadius, v) \big] &= \nabla L^\queryRadius(u), \\ 
    \label{Eqn: Prop, ZO grad, bound from smoothed version}
    \Vert \nabla L^\queryRadius(u) - F(u) \Vert_2 &\leq \ell \queryRadius, \\ 
    \label{Eqn: Prop, ZO grad, bound}
    \Vert \hat{F}(u; \queryRadius, v) \Vert_2 &\leq dG + \frac{dM_L}{\queryRadius}, \\
    \label{Eqn: Prop, ZO and true grad, difference}
    \Vert \hat{F}(u; \queryRadius, v) - F(u) \Vert &\leq \min\Bigg\{ (d+1) G + \frac{dM_L}{\queryRadius}, \ell \queryRadius + 2dG + \frac{2dM_L}{\queryRadius} \Bigg\}.
\end{align}
\end{proposition}

\begin{proof}
First, to establish \eqref{Eqn: Prop, ZO grad mean, equality}, observe that since $L^\queryRadius(u) = \E_{v \sim \Unif(B^d)}[L(u + \queryRadius v)]$ and $\hat{F}(u; \queryRadius, v) = \frac{d}{\queryRadius} \cdot L(u + \queryRadius v) v$ for each $u \in \R^d$, $\queryRadius > 0$, and $v \in \sphere^{d-1}$:
\begin{align} \nonumber
    \nabla L^\queryRadius (u) &= \nabla \E_{\overline{v} \sim \Unif(B^d)} \big[ L(u + \queryRadius \overline{v}) \big] \\ \nonumber
    &= \nabla \E_{\overline{v} \sim \Unif(\queryRadius \cdot B^d)} \big[ L(u + \overline{v}) \big] \\ \nonumber
    &= \frac{1}{\vol_d(\queryRadius \cdot B^d)} \cdot \nabla \left( \int_{\queryRadius \cdot B^d} L(u + \overline{v}) \hspace{0.5mm} d\overline{v} \right) \\ \label{Eqn: Stokes' Theorem, Application} 
    &= \frac{1}{\vol_d(\queryRadius \cdot B^d)} \cdot \int_{\queryRadius \cdot \sphere^{d-1}} L(u + v) \cdot \frac{v}{\Vert v \Vert_2} \hspace{0.5mm} dv, \\ \nonumber
    \E_{v \sim \Unif(\sphere^{d-1})} \big[ \hat{F}(u; \queryRadius, v) \big] &= \frac{d}{\queryRadius} \cdot \E_{v \sim \Unif(\sphere^{d-1})} \big[ L(u + \queryRadius v) v \big] \\ \nonumber
    &= \frac{d}{\queryRadius} \cdot \E_{v \sim \Unif(\queryRadius \cdot \sphere^{d-1})} \Bigg[ L(u + v) \cdot \frac{v}{\Vert v \Vert_2} \Bigg] \\ \nonumber
    &= \frac{d}{\queryRadius} \cdot \frac{1}{\vol_{d-1}(\queryRadius \cdot \sphere^{d-1})} \cdot \int_{\queryRadius \cdot \sphere^{d-1}} L(u + v) \cdot \frac{v}{\Vert v \Vert_2} \hspace{0.5mm} dv,
\end{align}
where \eqref{Eqn: Stokes' Theorem, Application} follows because Stokes' Theorem (see, e.g., Lee, Theorem 16.11 \cite{Lee2006IntroductionToSmoothManifolds}) implies that:
\begin{align*}
    \nabla \int_{\queryRadius \cdot B^d} L(u + \overline{v}) \hspace{0.5mm} d\overline{v} = \int_{\queryRadius \cdot \sphere^{d-1}} L(u + v) \cdot \frac{v}{\Vert v \Vert_2} \hspace{0.5mm} dv.
\end{align*}
The equality \eqref{Eqn: Prop, ZO grad mean, equality} now follows by observing that the surface-area-to-volume ratio of $\queryRadius \cdot B^d$ is $d/\queryRadius$. 

Next, to establish \eqref{Eqn: Prop, ZO grad, bound from smoothed version}, we note that:
\begin{align} \nonumber
    \Vert \nabla L^\queryRadius(u) - F(u) \Vert_2 &= \big\Vert \nabla \E_{\overline{v} \sim \Unif(B^d)} \big[ L^\queryRadius(u) - L(u) \big] \big\Vert_2 \\ \nonumber
    &= \frac{1}{\vol_d(B^d)} \cdot \Bigg\Vert \nabla \Bigg( \int_{B^d} \big[ L(u + \queryRadius \overline{v}) - L(u) \big] \hspace{0.5mm} d\overline{v} \Bigg) \Bigg\Vert_2 \\ \label{Eqn: Diff Under Int Sign}
    &\leq \frac{1}{\vol_d(B^d)} \cdot \Bigg\Vert \int_{B^d} \big[ F(u + \queryRadius \overline{v}) - F(u) \big] \hspace{0.5mm} d\overline{v} \Bigg\Vert_2 \\ \nonumber
    &\leq \frac{1}{\vol_d(B^d)} \cdot \int_{B^d} \big\Vert F(u + \queryRadius \overline{v}) - F(u) \big\Vert_2 \hspace{0.5mm} d\overline{v} \\ \nonumber
    &\leq \frac{1}{\vol_d(B^d)} \cdot \int_{B^d} \ell \queryRadius \cdot \Vert \overline{v} \Vert_2 \hspace{0.5mm} d\overline{v} \\ \nonumber
    &\leq \ell \queryRadius,
\end{align}
where \eqref{Eqn: Diff Under Int Sign} follows by differentiating under the integral sign (see, e.g., Rudin, Theorem 9.42 \cite{Rudin1976PrinciplesOfMathematicalAnalysis}), and the remaining inequalities follow from the fact that $F$ is $\ell$-Lipschitz.

Next, we establish \eqref{Eqn: Prop, ZO grad, bound} by using the triangle inequality and the $M_L$-boundedness of $L(\cdot)$ on $\X \times \Y$, and the $G$-Lipschitzness of $L(\cdot)$:
\begin{align*}
    |\hat{F}(u; \queryRadius, v)| &= \frac{d}{\queryRadius} |L(u + \queryRadius v)| \cdot \Vert v \Vert_2 \\
    &\leq \frac{d}{\queryRadius} \cdot \big( |L(u)| + |L(u + \queryRadius v) - L(u)| \big) \cdot 1 \\
    &\leq \frac{d}{\queryRadius} \cdot (M_L + \queryRadius G).
\end{align*}
We can then use \eqref{Eqn: Prop, ZO grad, bound} to establish \eqref{Eqn: Prop, ZO and true grad, difference} by observing that:
\begin{align*}
    |\hat{F}(u; \queryRadius, v) - F(u)| &\leq |\hat{F}(u; \queryRadius, v)| + |F(u)| \leq (d+1) G + \frac{dM_L}{\queryRadius}.
\end{align*}
and, from \eqref{Eqn: Prop, ZO grad, bound}:
\begin{align*}
    &|\hat{F}(u; \queryRadius, v) - F(u)| \\
    \leq \hspace{0.5mm} &\big|\hat{F}(u; \queryRadius, v) - \E_v[\hat{F}(u; \queryRadius, v)| u] \big| + \big|\E_v[\hat{F}(u; \queryRadius, v)| u] - F(u) \big| \\
    \leq \hspace{0.5mm} &\big|\hat{F}(u; \queryRadius, v) - \E_v[\hat{F}(u; \queryRadius, v)| u] \big| + \big|\nabla L^\queryRadius(u) - F(u) \big| \\
    \leq \hspace{0.5mm} &2 \left( dG + \frac{dM_L}{\queryRadius}  \right) + \ell \queryRadius
\end{align*}
This concludes the proof.
\end{proof}

Below, we present technical lemmas that allow us to analyze the convergence rate of the correlated iterates $\{u_i^t\}$ in our random reshuffling-based OGDA Algorithm (Alg. \ref{Alg: OGDA-RR}). 

Let $\sigma^0, \cdots, \sigma^{t-1}$ denote the permutations drawn from epoch 0 to epoch $t-1$, and let $\{u_i^t(\sigma^t)\}_{1 \leq i \leq n}$ and $\{u_i^t(\tilde\sigma^t)\}_{1 \leq i \leq n}$ denote the iterates obtained at epoch $t$, when the permutations $\sigma^t$ and $\tilde\sigma^t$ are used for the epoch \(t\), respectively. Moreover, let $\D_{i,t}$ denote the distribution of $\{u_i^t(\sigma^t)\}_{1 \leq i \leq n}$ under $\sigma^t$, and for \(1\leq r\leq n\) let $\D_{i,t}^{(r)}$ denote the distribution of $\{u_i^t(\sigma^t)\}_{1 \leq i \leq n}$ with $\sigma^t$ conditioned on the event $\{\sigma_{i-1}^t = r\}$. 

We use the \textit{$p$-Wasserstein distance between probability distributions on $\R^d$}, defined below, to characterize the distance between $\D_{i,t}$ and $\D_{i,t}^{(r)}$. This is used in the coupling-based techniques employed to establish non-asymptotic convergence results for our random reshuffling algorithm. 
{\color{black}Note the difference between the $p$-Wasserstein distance for probability distributions on $\R^d$, and the \textit{Wasserstein distance on $\mathcal{Z} := \R^d \times \{+1, -1\}$ associated with a metric $c: \mathcal{Z} \times \mathcal{Z} \ra [0, \infty)$}, defined in Appendix \ref{app: appendixProofofConvexConcavity} (Definition \ref{Def: Wasserstein Metric, with Cost}).}

\begin{definition}[\textbf{$p$-Wasserstein distance between distributions on $\R^d$}] \label{Def: p-Wasserstein Metric}
Let $\mu, \nu$ be probability distributions over $\R^d$ with finite $p$-th moments, for some $p \geq 1$, and let $\Pi(\mu, \nu)$ denote the set of all couplings (joint distributions) between $\mu$ and $\nu$. The $p$-\textit{Wasserstein distance between $\mu$ and $\nu$}, denoted $\W_p(\mu, \nu)$, is defined by:
\begin{align*}
    \W_p(\mu, \nu) = \inf_{(X, X') \sim \pi \in \Pi(\mu, \nu)} \Big( \E_\pi \big[ \Vert X - X' \Vert^p \big] \Big)^{1/p}.
\end{align*}
\end{definition}

The following proposition characterizes the 1-Wasserstein distance as a measure of the gap between Lipschitz functions of random variables.

\begin{proposition}[\textbf{Kantorovich Duality}] \label{Prop: Kantorovich Duality}
If $\mu, \nu$ are probability distributions over $\R^d$ with finite second moments, then:
\begin{align*}
    \W_1(\mu, \nu) = \sup_{g \in \emph{Lip(1)}} \E_{X \sim \mu}[g(X)] - \E_{Y \sim \nu}[g(Y)],
\end{align*}
where $\emph{Lip(1)} := \{g: \R^d \ra \R: g \text{ is 1-Lipschitz} \}$.
\end{proposition}

Using \cite[Lemma C.2]{YuMazumdar2021FastDistributionallyRobustLearning}, we now bound the difference between the unbiased gap $\E[\Delta(u_i^t)]$ and the biased gap $\E[L_{\sigma_i^t}(\x_{i+1}^t, y^\star) - L_{\sigma_i^t}(\x^\star, y_{i+1}^t)]$ using the Wasserstein metric.

\begin{lemma}\label{Lemma: App, Gap Difference}
Let $u^\star := (\x^\star, \y^\star) \in \R^{d_x} \times \R^{d_y} = \R^d$ denote a saddle point of the min-max optimization problem \eqref{Eqn:Min-maxProblem2}. Then, for each $t \in [T]$ and $i \in [n]$, the iterates $\{u_i^t\} = \{(\x_i^t, \y_i^t)\}$ of the OGDA-RR algorithm satisfy:
\begin{align*}
    \Big| \E[\Delta(u_{i+1}^t)] - \E\big[ L_{\sigma_i^t}(\x_{i+1}^t, y^\star) - L_{\sigma_i^t}(\x^\star, y_{i+1}^t) \big] \Big| \leq \frac{G}{n} \sum_{r=1}^n \W_2\big( \D_{i+1,t}, \D_{i+1,t}^r \big)
\end{align*}
\end{lemma}

\begin{proof}
Since $\sigma^t$ and $\tilde\sigma^t$ are independently generated permutations of $[n]$, the iterates $\{u_i^t\}_{1 \leq i \leq n} = \{u_i^t(\sigma^t)\}_{1 \leq i \leq n}$ and $\{u_i^t(\tilde\sigma^t)\}_{1 \leq i \leq n}$ are i.i.d. Thus, we have:
\begin{align*}
    \E[\Delta(u_{i+1}^t)] = \E\big[ L_{\sigma_i^t}(\x_{i+1}^t(\tilde\sigma^t), y^\star) - L_{\sigma_i^t}(\x^\star, \y_{i+1}^t(\tilde\sigma^t)) \big],
\end{align*}
and thus:
\begin{align} \nonumber
    &\Big| \E[\Delta(u_{i+1}^t)] - \E\big[ L_{\sigma_i^t}(\x_{i+1}^t, y^\star) - L_{\sigma_i^t}(\x^\star, y_{i+1}^t) \big] \Big| \\ \nonumber
    = \hspace{0.5mm} &\Big| \E\big[ L_{\sigma_i^t}(\x_{i+1}^t(\tilde\sigma^t), y^\star) - L_{\sigma_i^t}(\x^\star, \y_{i+1}^t(\tilde\sigma^t)) \big] - \E\big[ L_{\sigma_i^t}(\x_{i+1}^t, y^\star) - L_{\sigma_i^t}(\x^\star, y_{i+1}^t) \big] \Big| \\ \label{Eqn: App, conditional expectation}
    = \hspace{0.5mm} &\Bigg| \frac{1}{n} \sum_{r=1}^n \E\big[ L_r(\x_{i+1}^t(\tilde\sigma^t), y^\star) - L_r(\x^\star, \y_{i+1}^t(\tilde\sigma^t)) \big] \\ \nonumber
    &\hspace{1cm} - \frac{1}{n} \sum_{r=1}^n \E\big[ L_r(\x_{i+1}^t, y^\star) - L_r(\x^\star, y_{i+1}^t) \big| \sigma_i^t = r \big] \Bigg| \\ \nonumber
    \leq \hspace{0.5mm} & \frac{1}{n} \sum_{r=1}^n \Big| \E\big[ L_r(\x_{i+1}^t(\tilde\sigma^t), y^\star) - L_r(\x^\star, \y_{i+1}^t(\tilde\sigma^t)) \big] - \E\big[ L_r(\x_{i+1}^t, y^\star) - L_r(\x^\star, y_{i+1}^t) \big| \sigma_i^t = r \big] \Big| \\ \label{Eqn: App, Lipschitz-ness of L}
    \leq \hspace{0.5mm} & \frac{1}{n} \sum_{r=1}^n \sup_{g \in \text{Lip}(G)} \Big( \E\big[ g(\x_{i+1}^t(\tilde\sigma^t), \y_{i+1}^t(\tilde\sigma^t)) \big] - \E\big[ g(\x_{i+1}^t, \y_{i+1}^t) | \sigma_i^t = r \big] \Big) \\ \label{Eqn: App, Kantorovich, in proof}
    \leq \hspace{0.5mm} &\frac{1}{n} \sum_{r=1}^n G \cdot \W_1(\D_{i+1, t}, \D_{i+1, t}^{(r)}) \\ \label{Eqn: App, W1 <= W2}
    \leq \hspace{0.5mm} &\frac{1}{n} \sum_{r=1}^n G \cdot \W_2(\D_{i+1, t}, \D_{i+1, t}^{(r)}),
\end{align}
where \eqref{Eqn: App, conditional expectation} follows by properties of the conditional expectation on $\{\sigma_i^t = r\}$ and the fact that $\sigma^t$ and $\tilde\sigma^t$ are independent, \eqref{Eqn: App, Lipschitz-ness of L} follows from the fact that $L$ is Lipschitz, \eqref{Eqn: App, Kantorovich, in proof} follows from Proposition \ref{Prop: Kantorovich Duality}, and \eqref{Eqn: App, W1 <= W2} follows from the fact that $\W_1(\mu, \nu) \leq \W_2(\mu, \nu)$ for any two probability distributions $\mu, \nu$.
\end{proof}

The next lemma bounds the difference in the iterates $\{u_i^t(\sigma^t)\}$ and $\{u_i^t(\tilde\sigma^t)\}$ (assuming, as before, that $\sigma^0, \cdots, \sigma^{t-1}$ were fixed and identical for both sequences.)

\begin{lemma} \label{Lemma: App, u and u tilde bound}
Denote, with a slight abuse of notation, $u_i^t := u_i^t(\sigma^t)$ and $\tilde u_i^t := u_i^t(\tilde\sigma^t)$. Then:
\begin{align*}
    \Vert u_{i+1}^t - \tilde u_{i+1}^t \Vert_2 &\leq \Bigg(6nd + 14n + 2 \cdot \sum_{i=1}^{n}\textbf{\emph{1}}\{ \sigma_i^t \ne \tilde \sigma_i^t \} \Bigg) G \cdot \eta^t + 6nd M_L  \cdot \frac{\eta^t}{\queryRadius^t}.
\end{align*}
\end{lemma}

\begin{proof}
Our proof strategy is to bound the differences between zeroth-order and first-order OGDA updates, and between the OGDA and proximal point updates. To this end, we define:
\begin{align*}
    u_{i+1}^t &= \proj_{\X \times \Y} \Big(u_{i}^t - \eta^t \hat{F}_{\sigma_i^t}(u_{i}^t; \queryRadius^t,v_{i}^t)-\eta^t \hat{F}_{\sigma_{i-1}^t}(u_{i}^t;\queryRadius^t,v_{i}^t) 
    + \eta^t\hat{F}_{\sigma_{i-1}^t}(u_{i-1}^t;\queryRadius^t,v_{i-1}^t) \Big), \\
    \tilde u_{i+1}^t &= \proj_{\X \times \Y} \Big(\tilde u_{i}^t - \eta^t \hat{F}_{\tilde \sigma_i^t}(\tilde u_{i}^t; \queryRadius^t,v_{i}^t)-\eta^t \hat{F}_{\tilde \sigma_{i-1}^t}(\tilde  u_{i}^t;\queryRadius^t,v_{i}^t) 
    + \eta^t\hat{F}_{\tilde\sigma_{i-1}^t}(\tilde  u_{i-1}^t;\queryRadius^t,v_{i-1}^t) \Big), \\
    v_{i+1}^t &= \proj_{\X \times \Y} \Big(u_{i}^t - \eta^t F_{\sigma_i^t}(u_{i}^t)-\eta^t F_{\sigma_{i-1}^t}(u_{i}^t) 
    + \eta^t F_{\sigma_{i-1}^t}(u_{i-1}^t) \Big), \\
    \tilde v_{i+1}^t &= \proj_{\X \times \Y} \Big(\tilde u_{i}^t - \eta^t F_{\tilde \sigma_i^t}(\tilde  u_{i}^t)-\eta^t F_{\tilde \sigma_{i-1}^t}(\tilde u_{i}^t) 
    + \eta^t F_{\tilde \sigma_{i-1}^t}(\tilde u_{i-1}^t) \Big), \\
    w_{i+1}^t &= \proj_{\X \times \Y} \Big(u_i^t - \eta^t F_{\sigma_i^t}(w_{i+1}^t) \Big), \\
    \tilde w_{i+1}^t &= \proj_{\X \times \Y} \Big(\tilde u_i^t - \eta^t F_{\tilde \sigma_i^t}(\tilde w_{i+1}^t) \Big).
\end{align*}
By the triangle inequality:
\begin{align} \label{Eqn: App, u i+1 t bound, triangle inequality}
    \Vert u_{i+1}^t - \tilde u_{i+1}^t \Vert_2 &\leq \Vert u_{i+1}^t - v_{i+1}^t \Vert_2 + \Vert v_{i+1}^t - w_{i+1}^t \Vert_2 + \Vert w_{i+1}^t - \tilde w_{i+1}^t \Vert_2 \\ \nonumber
    &\hspace{1cm} + \Vert \tilde w_{i+1}^t - \tilde v_{i+1}^t \Vert_2 + \Vert \tilde v_{i+1}^t - \tilde u_{i+1}^t \Vert_2.
\end{align}
Observe that bounding the fourth term is equivalent to bounding the second term, and bounding the fifth term is equivalent to bounding the first term.

To bound the first term on the right hand side, we use Proposition \ref{Prop: App, ZO grad mean, variance} to conclude that:
\begin{align} \nonumber
    \Vert u_{i+1}^t - v_{i+1}^t \Vert_2 &\leq \eta^t \cdot \Vert \hat{F}_{\sigma_i^t}(u_i^t; \queryRadius^t, v_i^t) - F_{\sigma_i^t}(u_i^t) \Vert + \eta^t \cdot \Vert \hat{F}_{\sigma_{i-1}^t}(u_i^t; \queryRadius^t, v_i^t) - F_{\sigma_{i-1}^t}(u_i^t) \Vert \\ \nonumber
    &\hspace{1cm} + \eta^t \cdot \Vert \hat{F}_{\sigma_{i-1}^t}(u_{i-1}^t; \queryRadius^t, v_{i-1}^t) - F_{\sigma_{i-1}^t}(u_{i-1}^t) \Vert \\ \label{Eqn: App, u diff, 1}
    &\leq 3(d+1)G \eta^t + 3d M_L \cdot \frac{\eta^t}{\queryRadius^t}
\end{align}
For the second term, we use the $G$-Lipschitzness of $L_r$, for each $r \in [n]$ to conclude that:
\begin{align} \nonumber
    \Vert v_{i+1}^t - w_{i+1}^t \Vert_2 &\leq \eta^t \cdot |F_{\sigma_i^t}(u_i^t)| + \eta^t \cdot |F_{\sigma_{i-1}^t}(u_i^t)| + \eta^t \cdot |F_{\sigma_{i-1}^t}(u_{i-1}^t)| + \eta^t \cdot |F_{\sigma_i^t}(w_{i+1}^t)| \\ \label{Eqn: App, u diff, 2}
    &\leq 4G \cdot \eta^t.
\end{align}
For the third term, we observe that if $\sigma_i^t \ne \tilde\sigma_i^t$, then:
\begin{align} \nonumber
    \Vert w_{i+1}^t - \tilde w_{i+1}^t \Vert_2 &\leq \Vert u_i^t - \tilde u_i^t \Vert_2 + \eta^t \cdot \Vert F_{\sigma_i^t}(w_{i+1}^t) - F_{\tilde\sigma_i^t}(\tilde w_{i+1}^t) \Vert_2 \\ \label{Eqn: App, u diff, 3, 1}
    &\leq \Vert u_i^t - \tilde u_i^t \Vert_2 + 2G \cdot \eta^t.
\end{align}
On the other hand, if $\sigma_i^t = \tilde\sigma_i^t$, then:
\begin{align*}
    w_{i+1}^t &= \proj_{\X \times \Y} \Big(u_i^t - \eta^t F_{\sigma_i^t}(w_{i+1}^t) \Big), \\
    \tilde w_{i+1}^t &= \proj_{\X \times \Y} \Big(\tilde u_i^t - \eta^t F_{\sigma_i^t}(\tilde w_{i+1}^t) \Big),
\end{align*}
so we have:
\begin{align} \nonumber
    &\Vert w_{i+1}^t - \tilde w_{i+1}^t \Vert_2^2 \\ \label{Eqn: App, bound of w, Ineq, 1}
    \leq \hspace{0.5mm} &(w_{i+1}^t - \tilde w_{i+1}^t)^\top \big( (u_i^t - \eta \cdot F_{\sigma_i^t}(w_{i+1}^t)) - (\tilde u_i^t - \eta \cdot F_{\sigma_i^t}(\tilde w_{i+1}^t)) \big) \\ \nonumber
    = \hspace{0.5mm} &(w_{i+1}^t - \tilde w_{i+1}^t)^\top (u_i^t - \tilde u_i^t) - \eta(w_{i+1}^t - \tilde w_{i+1}^t)^\top \big( F_{\sigma_i^t}(w_{i+1}^t)) - F_{\sigma_i^t}(\tilde w_{i+1}^t)) \big) \\ \label{Eqn: App, bound of w, Ineq, 2}
    \leq \hspace{0.5mm} &(w_{i+1}^t - \tilde w_{i+1}^t)^\top (u_i^t - \tilde u_i^t) \\ \label{Eqn: App, u diff, 3, 2}
    \leq \hspace{0.5mm} &\Vert w_{i+1}^t - \tilde w_{i+1}^t \Vert_2 \cdot \Vert u_i^t - \tilde u_i^t \Vert_2,
\end{align}
so $\Vert w_{i+1}^t - \tilde w_{i+1}^t \Vert_2 \leq \Vert u_i^t - \tilde u_i^t \Vert_2$. Here, \eqref{Eqn: App, bound of w, Ineq, 1} follows from the definitions of $w_{i+1}^t$ and $\tilde w_{i+1}^t$, as well as Proposition \ref{Prop: App, Projection, 1}, while \eqref{Eqn: App, bound of w, Ineq, 2} holds because the monotonicity of $F_i$, for each $i \in [n]$, implies that $(w_{i+1}^t - \tilde w_{i+1}^t)^\top \big( F_{\sigma_i^t}(w_{i+1}^t) - F_{\sigma_i^t}(\tilde w_{i+1}^t) \big) \geq 0$. Putting together \eqref{Eqn: App, u diff, 1}, \eqref{Eqn: App, u diff, 2}, \eqref{Eqn: App, u diff, 3, 1}, \eqref{Eqn: App, u diff, 3, 2}, we have:
\begin{align*}
     \Vert u_{i+1}^t - \tilde u_{i+1}^t \Vert_2 &\leq \Vert u_i^t - \tilde u_i^t \Vert_2 + (6d + 14) G \cdot \eta^t + 6d M_L  \cdot \frac{\eta^t}{\queryRadius^t} \\
     &\hspace{1cm} + 2G \cdot \textbf{1}\{ \sigma_i^t \ne \tilde \sigma_i^t \} \cdot \eta^t,
\end{align*}
where the indicator $\textbf{1}(A)$ returns 1 if the given event $A$ occurs, and 0 otherwise. 

Since \(u_0^{t}=\td{u}_0^{t}\), we can iteratively apply the above inequality to obtain that, for any and epoch \(t\) and \(i\in [n]\): 
\begin{align*}
     \Vert u_{i+1}^t - \tilde u_{i+1}^t \Vert_2 &\leq  (6d + 14) nG \cdot \eta^t + 6nd M_L  \cdot \frac{\eta^t}{\queryRadius^t} +2\eta_tG \cdot \sum_{i=1}^{n}\textbf{1}\{ \sigma_i^t \ne \tilde \sigma_i^t \},
\end{align*}

\end{proof}

\begin{remark}
In the theorems and lemmas below, we will be concerned with the case where $\sigma^t$ and $\tilde\sigma^t$ have the following specific relationship. Let $\mathcal{R}_n$ denote the set of all random permutations over the set $[n]$. For each $l, m \in [n]$, let $S_{l, m}: \mathcal{R}_n \rightarrow \mathcal{R}_n$ denote the map that swaps,
for each input permutation $\sigma$, the $l$-th and $m$-th entries. For each $r,i \in [n]$, define the map $\omega_{r, i}: \mathcal{R}_n \rightarrow \mathcal{R}_n$ as follows: 
\begin{align*}
    \omega_{r, i}(\sigma) =
    \begin{cases}
    \sigma, & \text{if } \sigma_{i-1} = r, \\
    S_{i-1, j}(\sigma), & \textnormal{if } \sigma_j = r \textnormal{ and } j \neq i-1. 
    \end{cases}. 
\end{align*}
Intuitively, $\omega_{r, i}$ performs a single swap such that the $(i-1)$-th position of the permutation is $r$. Clearly, if $\sigma^t$ is a random permutation (i.e., selected from a uniform distribution over $\mathcal{R}_n$), then $\omega_{r, i}(\sigma^s)$ has the same distribution as $\sigma^t|(\sigma^t_{i-1} = r)$. Based on this construction, we have $u_i(\sigma^t) \sim \mathcal{D}_{i, t}$ and $u_{i}(\omega_{r,i}(\sigma^t)) \sim \mathcal{D}_{i, t}^{(r)}$. This gives a coupling between $\mathcal{D}_{s,t}$ and $\mathcal{D}_{s,t}^{(r)}$.
Since $\sigma^t$ and $\tilde \sigma^t$ differ by at most two entries, by iteratively applying Lemma \ref{Lemma: App, u and u tilde bound}, we have:
\begin{align*}
    \Vert u_{i+1}^t - \tilde u_{i+1}^t \Vert_2 &\leq n \Bigg( (6d + 14) G \cdot \eta^t + 6d M_L \cdot \frac{\eta^t}{\queryRadius^t} \Bigg) + 4G \cdot \eta^t \\
    &= (6nd + 14n + 4) G \cdot \eta^t + 6n d M_L \cdot \frac{\eta^t}{\queryRadius^t},
\end{align*}
as claimed.
\end{remark}
\begin{lemma} \label{Lemma: App, Inner Product of F}
If $\eta^t \leq 1/(2\ell)$ for each $t \in \{0, 1, \cdots, T-1\}$, the iterates $\{u_i^t\} = \{(\x_i^t, \y_i^t)\}$ of the OGDA-RR algorithm satisfy, for each $u \in \X \times \Y$:
\begin{align*}
    &2 \eta^t \cdot \E\Big[ \Big\langle F_{\sigma_i^t}(u_{i+1}^t), u_{i+1}^t - u \Big\rangle \Big] \\
    \leq \hspace{0.5mm} &\E\big[ \Vert u_i^t - u \Vert_2^2 \big] - \E\big[ \Vert u_{i+1}^t - u \Vert_2^2 \big] - \frac{1}{2} \E\big[ \Vert u_{i+1}^t - u_i^t \Vert_2^2 \big] + \frac{1}{2} \E\big[ \Vert u_i^t - u_{i-1}^t \Vert_2^2 \big] \\
    &\hspace{1cm} + 2 \eta^t \cdot \E\Big[ \Big\langle F_{\sigma_i^t}(u_{i+1}^t) - F_{\sigma_i^t}(u_i^t), u_{i+1}^t - u \Big\rangle \Big] \\
    &\hspace{1cm} - 2 \eta^t \cdot \E\Big[ \Big\langle F_{\sigma_{i-1}^t}(u_{i}^t) - F_{\sigma_{i-1}^t}(u_{i-1}^t), u_{i}^t - u \Big\rangle \Big] \\
    &\hspace{1cm} + 6C_1 \cdot \left(\eta^t \queryRadius^t + (\eta^t)^2 \queryRadius^t + (\eta^t)^2 + \frac{(\eta^t)^2}{\queryRadius^t} + \frac{(\eta^t)^2}{(\queryRadius^t)^2} \right),
\end{align*}
where $C_1 := d^2\max\big\{6G \ell D, 18G^2 + 6 M_L \ell D, 30 M_L G, 12 M_L^2 \big\}$ is a constant independent of the sequences $\{\eta^t\}$ and $\{\queryRadius^t\}$.
\end{lemma}

\begin{proof}
The iterates of the OGDA-RR algorithm are given by:
\begin{align} \nonumber
    u_{i+1}^t &= \proj_{\X \times \Y} \Big( u_i^t - \eta^t \hat{F}_{\sigma_i^t}(u_i^t; \queryRadius^t, v_i^t) - \eta^t \hat{F}_{\sigma_{i-1}^t}(u_i^t; \queryRadius^t, v_i^t) \\ \nonumber
    &\hspace{2cm} - \eta^t \hat{F}_{\sigma_{i-1}^t}(u_{i-1}^t; \queryRadius^t, v_{i-1}^t) \Big) \\ \label{Eqn: u in terms of gamma, E}
    &= \proj_{\X \times \Y} \Big( u_i^t - \eta^t F_{\sigma_i^t}(u_{i+1}^t) + \eta^t \big( \gamma_i^t + E_{i, 1}^t + E_{i, 2}^t + E_{i, 3}^t \big) \Big),
\end{align}
where we have defined:
\begin{align} 
    \nonumber \gamma_i^t &:= F_{\sigma_i^t }(u_{i+1}^t) -  F_{\sigma_i^t}(u_i^t)-F_{\sigma_{i-1}^t}(u_i^t)+F_{\sigma_{i-1}^t}(u_{i-1}^t), \\ 
    \nonumber
    E_{i, 1}^t &:= F_{\sigma_i^t}(u_i^t) - \hat{F}_{\sigma_i^t}(u_{i}^t;\queryRadius^t,v_i^t), \\
    \nonumber
    E_{i, 2}^t &:=  F_{\sigma_{i-1}^t}(u_i^t) - \hat{F}_{\sigma_{i-1}^t}(u_{i}^t;\queryRadius^t,v_i^t), \\
    \nonumber
    E_{i, 3}^t &:= F_{\sigma_{i-1}^t}(u_{i-1}^t)- \hat{F}_{\sigma_{i-1}^t}(u_{i-1}^t;\queryRadius^t,v_{i-1}^t).
\end{align}
First, by applying Lemma \ref{Lemma: Proj, applied to OGDA Thm}  we have:
\begin{align} \label{Eqn: OGDA Thm, Ineq 1}
    &2 \eta^t \cdot \E\Big[ \Big\langle F_{\sigma_i^t}(u_{i+1}^t), u_{i+1}^t - u \Big\rangle \Big] \\ \nonumber
    \leq \hspace{0.5mm} &\E\big[ \Vert u_i^t - u \Vert_2^2 \big] - \E\big[ \Vert u_{i+1}^t - u \Vert_2^2 \big] - \E\big[ \Vert u_{i+1}^t - u_i^t \Vert_2^2 \big] \\ \nonumber
    &\hspace{1cm} + 2 \eta^t \cdot \E\Big[ \Big\langle \gamma_i^t, u_{i+1}^t - u \Big\rangle \Big] + \sum_{k=1}^3 2 \eta^t \cdot \E\Big[ \Big\langle E_{i,k}^t, u_{i+1}^t - u \Big\rangle \Big].
\end{align}
Below, we proceed to bound the inner product terms on the right-hand-side of \eqref{Eqn: OGDA Thm, Ineq 1}. First, we bound $\big\langle \gamma_i^t, u_{i+1}^t - u \big\rangle$:
\begin{align} \nonumber
    \Big\langle \gamma_i^t, u_{i+1}^t - u \Big\rangle &= \Big\langle F_{\sigma_i^t}(u_{i+1}^t) - F_{\sigma_i^t}(u_i^t), u_{i+1}^t - u \Big\rangle \\ \nonumber
    &\hspace{1cm} - \Big\langle F_{\sigma_{i-1}^t}(u_i^t) - F_{\sigma_{i-1}^t}(u_{i-1}^t), u_{i+1}^t - u \Big\rangle \\ \nonumber
    &= \Big\langle F_{\sigma_i^t}(u_{i+1}^t) - F_{\sigma_i^t}(u_i^t), u_{i+1}^t - u \Big\rangle \\ \nonumber
    &\hspace{1cm} - \Big\langle F_{\sigma_{i-1}^t}(u_i^t) - F_{\sigma_{i-1}^t}(u_{i-1}^t), u_i^t - u \Big\rangle \\ \nonumber
    &\hspace{1cm} - \Big\langle F_{\sigma_{i-1}^t}(u_i^t) - F_{\sigma_{i-1}^t}(u_{i-1}^t), u_{i+1}^t - u_i^t \Big\rangle \\ \label{Eqn: OGDA Thm, gamma i t}
    &\leq \Big\langle F_{\sigma_i^t}(u_{i+1}^t) - F_{\sigma_i^t}(u_i^t), u_{i+1}^t - u \Big\rangle \\ \nonumber
    &\hspace{1cm} - \Big\langle F_{\sigma_{i-1}^t}(u_i^t) - F_{\sigma_{i-1}^t}(u_{i-1}^t), u_i^t - u \Big\rangle \\ \nonumber
    &\hspace{1cm} + \frac{1}{2} \ell \cdot \Vert u_i^t - u_{i-1}^t \Vert_2^2 + \frac{1}{2}\ell \cdot \Vert u_{i+1}^t - u_i^t \Vert_2^2.
\end{align}
Note that the final inequality follows by applying Young's inequality, and noting that $F$ is $\ell$-Lipschitz. Next, we bound $\langle E_{i,1}^t, u_{i+1}^t - u \rangle$:
\begin{align} \nonumber
    &\E\big[ \langle E_{i,1}^t, u_{i+1}^t - u \rangle \Big] \\ \nonumber
    = \hspace{0.5mm} &\E\Big[ \Big\langle F_{\sigma_i^t}(u_i^t) - \hat{F}_{\sigma_i^t}(u_i^t, \queryRadius^t, v_i^t), u_{i+1}^t - u \Big\rangle \Big] \\ \nonumber
    = \hspace{0.5mm} &\E\Big[ \Big\langle F_{\sigma_i^t}(u_i^t) - \nabla L_{\sigma_i^t}^{\queryRadius^t}(u_i^t), u_{i+1}^t - u \Big\rangle \Big] \\ \nonumber
    &\hspace{1cm} + \E\Big[ \Big\langle \E\big[ \hat{F}_{\sigma_i^t}(u_i^t; \queryRadius^t, v_i^t | u_i^t \big] - \hat{F}_{\sigma_i^t}(u_i^t, \queryRadius^t, v_i^t), u_{i+1}^t - u \Big\rangle \Big] \\ \label{Eqn: OGDA Thm, E i t, 3 terms}
    = \hspace{0.5mm} &\E\Big[ \Big\langle F_{\sigma_i^t}(u_i^t) - \nabla L_{\sigma_i^t}^{\queryRadius^t}(u_i^t), u_{i+1}^t - u \Big\rangle \Big] \\ \nonumber
    &\hspace{1cm} + \E\Big[ \Big\langle \E_v\big[ \hat{F}_{\sigma_i^t}(u_i^t; \queryRadius^t, v | u_i^t \big] - \hat{F}_{\sigma_i^t}(u_i^t, \queryRadius^t, v_i^t), u_i^t - u \Big\rangle \Big] \\ \nonumber
    &\hspace{1cm} + \E\Big[ \Big\langle \E_v\big[ \hat{F}_{\sigma_i^t}(u_i^t; \queryRadius^t, v | u_i^t \big] - \hat{F}_{\sigma_i^t}(u_i^t, \queryRadius^t, v_i^t), u_{i+1}^t - u_i^t \Big\rangle \Big],
\end{align}
where the first equality above follows by applying Proposition \ref{Prop: App, ZO grad mean, variance}, \eqref{Eqn: Prop, ZO grad mean, equality}, and we have used the shorthand $\E_v := \E_{v \sim \Unif(\sphere^{d-1})}$. (Recall that $L^\queryRadius(u) := \E_{v \sim \Unif(\sphere^{d-1})}\big[ L(u + \queryRadius v) \big]$) Next, we upper bound each of the three quantities in \eqref{Eqn: OGDA Thm, E i t, 3 terms}. First, by Proposition \ref{Prop: App, ZO grad mean, variance}, \eqref{Eqn: Prop, ZO grad, bound from smoothed version}, we have:
\begin{align} \nonumber
    &\E\Big[ \Big\langle F_{\sigma_i^t}(u_i^t) - \nabla L_{\sigma_i^t}^{\queryRadius^t}(u_i^t), u_{i+1}^t - u \Big\rangle \Big] \\ \nonumber
    \leq \hspace{0.5mm} &\E\Big[ \Vert F_{\sigma_i^t}(u_i^t) - \nabla L_{\sigma_i^t}^{\queryRadius^t}(u_i^t) \Vert_2 \cdot \Vert u_{i+1}^t - u \Vert_2 \Big] \\ \label{Eqn: OGDA Thm, Term 1/3}
    \leq \hspace{0.5mm} &\ell D \cdot \queryRadius^t,
\end{align}
with $C_1 > 0$ as given in Lemma \ref{Lemma: App, Inner Product of F}. Meanwhile, the law of iterated expectations can be used to bound the second quantity:
\begin{align} \nonumber
    &\E\Big[ \Big\langle \E_v\big[ \hat{F}_{\sigma_i^t}(u_i^t; \queryRadius^t, v) | u_i^t \big] - \hat{F}_{\sigma_i^t}(u_i^t, \queryRadius^t, v_i^t), u_i^t - u \Big\rangle \Big] \\
    \nonumber
    = \hspace{0.5mm} &\E\Big[ \E_v\Big[ \Big\langle \hat{F}_{\sigma_i^t}(u_i^t, \queryRadius^t, v_i^t), u_i^t - u \Big\rangle \big| u_i^t \Big] \Big] - \E\Big[ \Big\langle \hat{F}_{\sigma_i^t}(u_i^t, \queryRadius^t, v_i^t), u_i^t - u \Big\rangle \Big] \\ \label{Eqn: OGDA Thm, Term 2/3}
    = \hspace{0.5mm} &0,
\end{align}
and we can upper-bound the third quantity as shown below. By using the compactness of $\X \times \Y$ and the continuity of $L$, we have:
\begin{align} \nonumber
    &\E\Big[ \Big\langle \E_v\big[ \hat{F}_{\sigma_i^t}(u_i^t; \queryRadius^t, v) | u_i^t \big] - \hat{F}_{\sigma_i^t}(u_i^t, \queryRadius^t, v_i^t), u_{i+1}^t - u_i^t \Big\rangle \Big] \\ \nonumber
    \leq \hspace{0.5mm} & \Big( \big\Vert \E_v\big[ \hat{F}_{\sigma_i^t}(u_i^t; \queryRadius^t, v)| u_i^t \big] \big\Vert_2 + \Vert \hat{F}_{\sigma_i^t}(u_i^t, \queryRadius^t, v_i^t) \Vert \Big) \cdot \Vert u_{i+1}^t - u_i^t \Vert_2 \\ \nonumber
    \leq \hspace{0.5mm} &2 \cdot \frac{{d}}{\queryRadius^t} \cdot \sup_{\stackrel{u \in \X \times \Y}{v \sim \Unif(\sphere^{d-1})}} |L(u_i^t + \queryRadius^t v)| \cdot \Vert u_{i+1}^t - u_i^t \Vert_2, \\ \label{Eqn: OGDA Thm, Term 3/3, Part 1}
    \leq \hspace{0.5mm} &2 \cdot \frac{{d}}{\queryRadius^t} \cdot (M_L + \queryRadius^t G) \cdot \Vert u_{i+1}^t - u_i^t \Vert_2,
\end{align}
and using \eqref{Eqn: OGDA Thm, Term 1/3} and the bound for each $\Vert \hat{F}_{\sigma_i^t} \Vert_2$ given in \eqref{Eqn: OGDA Thm, Term 3/3, Part 1}, we have:
\begin{align} \nonumber
    &\Vert u_{i+1}^t - u_i^t \Vert_2 \\ \nonumber
    \leq \hspace{0.5mm} &\eta^t \cdot \Vert \hat{F}_{\sigma_i^t}(u_i^t; \queryRadius^t, v_i^t) + \hat{F}_{\sigma_{i-1}^t}(u_i^t; \queryRadius^t, v_i^t) - \hat{F}_{\sigma_{i-1}^t}(u_{i-1}^t; \queryRadius^t, v_{i-1}^t) \Vert \\ \nonumber
    \leq \hspace{0.5mm} &\eta^t \cdot \Vert F_{\sigma_i^t}(u_i^t) + F_{\sigma_{i-1}^t}(u_i^t) - F_{\sigma_{i-1}^t}(u_{i-1}^t) \Vert_2 \\ \nonumber
    &\hspace{5mm} + \eta^t{\color{black}d} \cdot \Vert \hat{F}_{\sigma_i^t}(u_i^t; \queryRadius^t, v_i^t) - F_{\sigma_i^t}(u_i^t) \Vert_2 \\ \nonumber
    &\hspace{5mm} + \eta^t{\color{black}d} \cdot \Vert \hat{F}_{\sigma_{i-1}^t}(u_i^t; \queryRadius^t, v_i^t) - F_{\sigma_{i-1}^t}(u_i^t) \Vert_2 \\ \nonumber
    &\hspace{5mm} + \eta^t {\color{black}d}\cdot \Vert \hat{F}_{\sigma_{i-1}^t}(u_{i-1}^t; \queryRadius^t, v_{i-1}^t) - F_{\sigma_{i-1}^t}(u_{i-1}^t) \Vert_2 \\ \label{Eqn: OGDA Thm, Term 3/3, Part 2}
    \leq \hspace{0.5mm} &3G \eta^t + 3 \eta^t{\color{black}d} \cdot \Bigg( 2(M_L + G \queryRadius^t) \cdot
    \frac{1}{\queryRadius^t} + \ell D \cdot \queryRadius^t \Bigg).
\end{align}
Substituting \eqref{Eqn: OGDA Thm, Term 3/3, Part 2} back into \eqref{Eqn: OGDA Thm, Term 3/3, Part 1}, we have:
\begin{align} \nonumber
    &\E\Big[ \Big\langle \E_v\big[ \hat{F}_{\sigma_i^t}(u_i^t; \queryRadius^t, v) | u_i^t \big] - \hat{F}_{\sigma_i^t}(u_i^t, \queryRadius^t, v_i^t), u_{i+1}^t - u_i^t \Big\rangle \Big] \\ \nonumber
    \leq \hspace{0.5mm} &{\color{black}d^2}\ell D 6 \eta^t G \cdot \queryRadius^t + 6{\color{black}d^2} \eta^t (3G^2 + M_L \ell D) + 30{\color{black}d^2} \eta^t M_L G \cdot \frac{1}{\queryRadius^t} + 12{\color{black}d^2} \eta^t M_L^2 \cdot \left( \frac{
    1}{\queryRadius^t} \right)^2 \\  \label{Eqn: OGDA Thm, Term 3/3}
    \leq \hspace{0.5mm} &C_1 \cdot \left( \eta^t \queryRadius^t + \eta^t + \frac{\eta^t}{\queryRadius^t} + \frac{\eta^t}{(\queryRadius^t)^2} \right),
\end{align}
where $C_1 := {\color{black}d^2} \cdot \max\big\{6G \ell D, 18G^2 + 6 M_L \ell D, 30 M_L G, 12 M_L^2 \big\}$ is a constant independent of the sequences $\{\eta^t\}$ and $\{\queryRadius^t\}$. The quantities $\E\big[ \langle E_{i,2}^t, u_{i+1}^t - u \rangle \Big]$ and $\E\big[ \langle E_{i,3}^t, u_{i+1}^t - u \rangle \Big]$ can be similarly bounded. Substituting \eqref{Eqn: OGDA Thm, Term 1/3}, \eqref{Eqn: OGDA Thm, Term 2/3}, \eqref{Eqn: OGDA Thm, Term 3/3} back into \eqref{Eqn: OGDA Thm, E i t, 3 terms}, and substituting \eqref{Eqn: OGDA Thm, E i t, 3 terms} and \eqref{Eqn: OGDA Thm, gamma i t} into \eqref{Eqn: OGDA Thm, Ineq 1}, we find that:
\begin{align*}
    &2 \eta^t \cdot \E\Big[ \Big\langle F_{\sigma_i^t}(u_{i+1}^t), u_{i+1}^t - u \Big\rangle \Big] \\
    = \hspace{0.5mm} &\E\big[ \Vert u_i^t - u \Vert_2^2 \big] - \E\big[ \Vert u_{i+1}^t - u \Vert_2^2 \big] - \E\big[ \Vert u_{i+1}^t - u_i^t \Vert_2^2 \big] \\
    &\hspace{1cm} + 2 \eta^t \cdot \E\Big[ \Big\langle \gamma_i^t, u_{i+1}^t - u \Big\rangle \Big] + 2 \eta^t \cdot \sum_{k=1}^3 \E\Big[ \Big\langle E_{i,k}^t, u_{i+1}^t - u \Big\rangle \Big] \\
    \leq \hspace{0.5mm} &\E\big[ \Vert u_i^t - u \Vert_2^2 \big] - \E\big[ \Vert u_{i+1}^t - u \Vert_2^2 \big] - \E\big[ \Vert u_{i+1}^t - u_i^t \Vert_2^2 \big] \\
    &\hspace{1cm} + 2 \eta^t \cdot \E\Big[ \Big\langle F_{\sigma_i^t}(u_{i+1}^t) - F_{\sigma_i^t}(u_i^t), u_{i+1}^t - u \Big\rangle \Big] \\
    &\hspace{1cm} - 2 \eta^t \cdot \E\Big[ \Big\langle F_{\sigma_{i-1}^t}(u_i^t) - F_{\sigma_{i-1}^t}(u_{i-1}^t), u_i^t - u \Big\rangle \Big] \\
    &\hspace{1cm} + \eta^t \ell \cdot \E\big[ \Vert u_i^t - u_{i-1}^t \Vert_2^2 \big] + \eta^t \ell \cdot \E\big[  \Vert u_{i+1}^t - u_i^t \Vert_2^2 \big] \\
    &\hspace{1cm} + 6C_1 \cdot \left(\eta^t \queryRadius^t + (\eta^t)^2 \queryRadius^t + (\eta^t)^2 + \frac{(\eta^t)^2}{\queryRadius^t} + \frac{(\eta^t)^2}{(\queryRadius^t)^2} \right),
\end{align*}
In particular, since by assumption $\eta^t \leq 1/(2\ell)$ for each $t \in \{0, 1, \cdots, T-1\}$, then:
\begin{align*}
    &2 \eta^t \cdot \E\Big[ \Big\langle F_{\sigma_i^t}(u_{i+1}^t), u_{i+1}^t - u \Big\rangle \Big] \\
    \leq \hspace{0.5mm} &\E\big[ \Vert u_i^t - u \Vert_2^2 \big] - \E\big[ \Vert u_{i+1}^t - u \Vert_2^2 \big] - \frac{1}{2} \E\big[ \Vert u_{i+1}^t - u_i^t \Vert_2^2 \big] + \frac{1}{2} \E\big[ \Vert u_i^t - u_{i-1}^t \Vert_2^2 \big] \\
    &\hspace{1cm} + 2 \eta^t \cdot \E\Big[ \Big\langle F_{\sigma_i^t}(u_{i+1}^t) - F_{\sigma_i^t}(u_i^t), u_{i+1}^t - u \Big\rangle \Big] \\
    &\hspace{1cm} - 2 \eta^t \cdot \E\Big[ \Big\langle F_{\sigma_{i-1}^t}(u_i^t) - F_{\sigma_{i-1}^t}(u_{i-1}^t), u_i^t - u \Big\rangle \Big] \\
    &\hspace{1cm} + 6C_1 \cdot \left(\eta^t \queryRadius^t + (\eta^t)^2 \queryRadius^t + (\eta^t)^2 + \frac{(\eta^t)^2}{\queryRadius^t} + \frac{(\eta^t)^2}{(\queryRadius^t)^2} \right),
\end{align*}

\end{proof}

Finally, to bound the step size terms above, we require the following lemma, which follows from standard calculus arguments.

\begin{lemma} \label{Lemma: App, Geometric Sum Bound}
\begin{alignat*}{2}
    \sum_{t=1}^T t^{-\beta} &\geq \frac{1}{1-\beta} T^{1-\beta}, \hspace{1cm} &&\forall \hspace{0.5mm} \beta < 1, \\
    \sum_{t=1}^T t^{-(1+\beta)} &\leq \frac{1}{\beta} + 1, &&\forall \hspace{0.5mm} \beta > 0.
\end{alignat*}
\end{lemma}


\subsection{Proof of Theorem \ref{Thm: OGDA Convergence}}
\label{subsec: App, Proof of Theorem 4.1}

\begin{proof}(\textbf{Proof of Theorem \ref{Thm: OGDA Convergence}})
By applying Lemma \ref{Lemma: App, Inner Product of F} (note that $\eta^t \leq \eta^0 \leq \frac{1}{2\ell}$, for each $t \in \{0, 1, \cdots, T-1\}$) and using convex-concave nature of \(L_r\) (refer Proposition 1 in \cite{MokhtariOzdaglar2020ConvergenceRO}), for each $r \in \{1, \cdots n\}$, we have:
\begin{align} \nonumber
    &2 \eta^t \cdot \E\big[ L_{\sigma_i^t}(\x_{i+1}^t, y^\star) - L_{\sigma_i^t}(\x^\star, \y_{i+1}^t) \big] \\ \nonumber
    \leq \hspace{0.5mm} &2 \eta^t \cdot \E\big[ \big\langle F_{\sigma_i^t}(u_{i+1}^t), u_{i+1}^t - u^\star \big\rangle \big] \\ \nonumber
    \leq \hspace{0.5mm} &\E\big[ \Vert u_i^t - u^\star \Vert_2^2 \big] - \E\big[ \Vert u_{i+1}^t - u^\star \Vert_2^2 \big] - \frac{1}{2} \E\big[ \Vert u_{i+1}^t - u_i^t \Vert_2^2 \big] + \frac{1}{2} \E\big[ \Vert u_i^t - u_{i-1}^t \Vert_2^2 \big] \\ \nonumber
    &\hspace{1cm} + 2 \eta^t \cdot \E\Big[ \Big\langle F_{\sigma_i^t}(u_{i+1}^t) - F_{\sigma_i^t}(u_i^t), u_{i+1}^t - u^\star \Big\rangle \Big] \\ \nonumber
    &\hspace{1cm} - 2 \eta^t \cdot \E\Big[ \Big\langle F_{\sigma_{i-1}^t}(u_i^t) - F_{\sigma_{i-1}^t}(u_{i-1}^t), u_i^t - u^\star \Big\rangle \Big] \\ \label{Eqn: App, OGDA Thm proof, 1}
    &\hspace{1cm} + 6C_1 \cdot \left(\eta^t \queryRadius^t + (\eta^t)^2 \queryRadius^t + (\eta^t)^2 + \frac{(\eta^t)^2}{\queryRadius^t} + \frac{(\eta^t)^2}{(\queryRadius^t)^2} \right).
\end{align}
Meanwhile, Lemma \ref{Lemma: App, Gap Difference}, Proposition \ref{Prop: Kantorovich Duality} (Kantorovich Duality), and Lemma \ref{Lemma: App, u and u tilde bound} imply that:
\begin{align*}
    &\Big| \E[\Delta(u_{i+1}^t)] - \E\big[ L_{\sigma_i^t}(\x_{i+1}^t, y^\star) - L_{\sigma_i^t}(\x^\star, y_{i+1}^t) \big] \Big| \leq \frac{G}{n} \sum_{r=1}^n \W_2\big( \D_{i+1,t}, \D_{i+1,t}^r \big) \\
    &\hspace{1cm} \leq \frac{G}{n} \sum_{r=1}^n \sqrt{\E\big[ \big\Vert u_{i+1}^t(\sigma^t) - u_{i+1}^t(\tilde\sigma^t) \big\Vert_2^2 \big]} \\
    &\hspace{1cm} \leq G \cdot \Bigg( (6nd + 14n + 4) G \cdot \eta^t + 6n d M_L \cdot \frac{\eta^t}{\queryRadius^t} \Bigg).
\end{align*}

Substituting back into \eqref{Eqn: App, OGDA Thm proof, 1}, we have:
\begin{align} \nonumber
    &2 \eta^t \cdot \E\big[ \Delta(u_i^t) \big] \\ \nonumber
    \leq \hspace{0.5mm} &2 \eta^t \cdot \E\big[ L_{\sigma_i^t}(\x_{i+1}^t, y^\star) - L_{\sigma_i^t}(\x^\star, \y_{i+1}^t) \big] \\ \nonumber
    &\hspace{1cm} +  G \cdot \Bigg( (12nd + 28n + 8) G \cdot (\eta^t)^2 + 12n d M_L \cdot \frac{(\eta^t)^2}{\queryRadius^t} \Bigg) \\ \nonumber
    \leq \hspace{0.5mm} &\E\big[ \Vert u_i^t - u^\star \Vert_2^2 \big] - \E\big[ \Vert u_{i+1}^t - u^\star \Vert_2^2 \big] - \frac{1}{2} \E\big[ \Vert u_{i+1}^t - u_i^t \Vert_2^2 \big] + \frac{1}{2} \E\big[ \Vert u_i^t - u_{i-1}^t \Vert_2^2 \big] \\ \nonumber
    &\hspace{1cm} + 2 \eta^t \cdot \E\Big[ \Big\langle F_{\sigma_i^t}(u_{i+1}^t) - F_{\sigma_i^t}(u_i^t), u_{i+1}^t - u \Big\rangle \Big] \\ \nonumber
    &\hspace{1cm} - 2 \eta^t \cdot \E\Big[ \Big\langle F_{\sigma_{i-1}^t}(u_i^t) - F_{\sigma_{i-1}^t}(u_{i-1}^t), u_i^t - u \Big\rangle \Big] \\ \nonumber 
    &\hspace{1cm} + 6C_1 \cdot \left(\eta^t \queryRadius^t + (\eta^t)^2 \queryRadius^t + (\eta^t)^2 + \frac{(\eta^t)^2}{\queryRadius^t} + \frac{(\eta^t)^2}{(\queryRadius^t)^2} \right) \\ \label{Eqn: App, OGDA Thm proof, 2}
    &\hspace{1cm} +  G \cdot \Bigg( (12nd + 28n + 8) G \cdot (\eta^t)^2 + 12n d M_L \cdot \frac{(\eta^t)^2}{\queryRadius^t} \Bigg).
\end{align}
We can now sum the above telescoping terms across the $t$-th epoch, as shown below:
\begin{align} \nonumber
    &2 \cdot \sum_{i=1}^n \eta^t \cdot \E\big[ \Delta(u_i^t) \big] \\ \nonumber
    \leq \hspace{0.5mm} &\E\big[ \Vert u_1^t - u^\star \Vert_2^2 \big] - \E\big[ \Vert u_1^{t+1} - u^\star \Vert_2^2 \big] + \frac{1}{2} \E\big[ \Vert u_1^t - u_0^t \Vert_2^2 \big] - \frac{1}{2} \E\big[ \Vert u_1^{t+1} - u_0^{t+1} \Vert_2^2 \big] \\ \nonumber
    &\hspace{1cm} + 2 \eta^t \cdot \E\Big[ \Big\langle F_{\sigma_0^t}(u_1^t) - F_{\sigma_0^t}(u_0^t), u_1^t - u^\star \Big\rangle \Big] \\ \nonumber
    &\hspace{1cm} - 2 \eta^t \cdot \E\Big[ \Big\langle F_{\sigma_0^{t+1}}(u_1^{t+1}) - F_{\sigma_0^{t+1}}(u_0^{t+1}), u_1^{t+1} - u^\star \Big\rangle \Big] \\ \nonumber 
    &\hspace{1cm} + 6nC_1 \cdot \left(\eta^t \queryRadius^t + (\eta^t)^2 \queryRadius^t + (\eta^t)^2 + \frac{(\eta^t)^2}{\queryRadius^t} + \frac{(\eta^t)^2}{(\queryRadius^t)^2} \right) \\ \nonumber
    &\hspace{1cm} +  nG \cdot \Bigg( (12nd + 28n + 8) G \cdot (\eta^t)^2 + 12n d M_L \cdot \frac{(\eta^t)^2}{\queryRadius^t} \Bigg).
\end{align}
Meanwhile, we have for each $t = 0, 1, \cdots, T-1$, $i \in [n]$:
\begin{align*}
    &\E\Big[ \Big\langle F_{\sigma_i^t}(u_{i+1}^t) - F_{\sigma_i^t}(u_i^t), u_{i+1}^t - u^\star \Big\rangle \Big] \\
    \leq \hspace{0.5mm} &\E\big[ \big\Vert F_{\sigma_i^t}(u_{i+1}^t) - F_{\sigma_i^t}(u_i^t) \big\Vert \cdot \big\Vert u_{i+1}^t - u^\star \big\Vert \Big] \\
    = \hspace{0.5mm} &\ell \cdot \E\big[ \Vert u_{i+1}^t - u_i^t \Vert \big] \cdot D \\
    \leq \hspace{0.5mm} &\ell D \cdot \E\Big[ \big\Vert - \eta^t \hat{F}_{\sigma_i^t}(u_{i}^t; \queryRadius^t,v_{i}^t)-\eta^t \hat{F}_{\sigma_{i-1}^t}(u_{i}^t;\queryRadius^t,v_{i}^t) + \eta^t\hat{F}_{\sigma_{i-1}^t}(u_{i-1}^t;\queryRadius^t,v_{i-1}^t) \big\Vert \Big] \\
    \leq \hspace{0.5mm} &3\ell D \cdot \eta^t \cdot \Bigg( dG + \frac{dM_L}{\queryRadius^t} \Bigg) \\
    = \hspace{0.5mm} &3 \ell D dG \cdot \eta^t + 3 \ell D d M_L \cdot \frac{\eta^t}{\queryRadius^t},
\end{align*}
where the final inequality follows from Proposition \ref{Prop: App, ZO grad mean, variance}, \eqref{Eqn: Prop, ZO grad, bound}. We can upper bound $\E\Big[ \Big\langle F_{\sigma_{i-1}^t}(u_i^t) - F_{\sigma_{i-1}^t}(u_{i-1}^t), u_i^t - u \Big\rangle \Big]$ in a similar fashion. Substituting back into \eqref{Eqn: App, OGDA Thm proof, 2}, we have:
\begin{align} \nonumber
    &2 \cdot \sum_{i=1}^n \eta^t \cdot \E\big[ \Delta(u_i^t) \big] \\ \nonumber
    \leq \hspace{0.5mm} &\E\big[ \Vert u_1^t - u^\star \Vert_2^2 \big] - \E\big[ \Vert u_1^{t+1} - u^\star \Vert_2^2 \big] + \frac{1}{2} \E\big[ \Vert u_1^t - u_0^t \Vert_2^2 \big] - \frac{1}{2} \E\big[ \Vert u_1^{t+1} - u_0^{t+1} \Vert_2^2 \big] \\ \nonumber
    &\hspace{1cm} + 6nC_1 \cdot \left(\eta^t \queryRadius^t + (\eta^t)^2 \queryRadius^t + (\eta^t)^2 + \frac{(\eta^t)^2}{\queryRadius^t} + \frac{(\eta^t)^2}{(\queryRadius^t)^2} \right) \\ \nonumber
    &\hspace{1cm} + nG \cdot \Bigg( (12nd + 28n + 8) G \cdot (\eta^t)^2 + 12n d M_L \cdot \frac{(\eta^t)^2}{\queryRadius^t} \Bigg) \\ \nonumber
    &\hspace{1cm} + 6 \ell D dG \cdot (\eta^t)^2 + 6 \ell D d M_L \cdot \frac{(\eta^t)^2}{\queryRadius^t} \\
    \leq \hspace{0.5mm} &\E\big[ \Vert u_1^t - u^\star \Vert_2^2 \big] - \E\big[ \Vert u_1^{t+1} - u^\star \Vert_2^2 \big] + \frac{1}{2} \E\big[ \Vert u_1^t - u_0^t \Vert_2^2 \big] - \frac{1}{2} \E\big[ \Vert u_1^{t+1} - u_0^{t+1} \Vert_2^2 \big] \\ \nonumber
    &\hspace{1cm} + 2C \cdot \left(\eta^t \queryRadius^t + (\eta^t)^2 \queryRadius^t + (\eta^t)^2 + \frac{(\eta^t)^2}{\queryRadius^t} + \frac{(\eta^t)^2}{(\queryRadius^t)^2} \right),
\end{align}
where $C :=  \max\{3nC_1, (6nd + 14n + 4)nG, 6ndM_L, 3\ell D dG, 3\ell D dM_L \}$.

Finally, summing the above telescoping terms over $i \in [n]$ and $t \in \{0, 1, \cdots, T-1\}$, and removing non-positive terms, we obtain:
\begin{align} \nonumber
    &\frac{\sum_{t=0}^{T-1} \sum_{i=1}^n \eta^t \cdot \E\big[\Delta(u_i^t) \big]}{\sum_{t=0}^{T-1} \sum_{i=1}^n \eta^t} \\ \nonumber
    \leq \hspace{0.5mm} &\frac{1}{2 \cdot \sum_{t=0}^{T-1} \sum_{i=1}^n \eta^t} \Bigg( \Vert u_0^0 - u^\star \Vert_2 - \E\big[ \Vert u_n^{T-1} - u^\star \Vert_2 \big] + \frac{1}{2} \Vert u_1^0 - u_0^0 \Vert_2 - \frac{1}{2} \E\big[ \Vert u_n^{T-1} - u_{n-1}^{T-1} \Vert_2 \big] \Bigg) \\ \nonumber
    &\hspace{1cm} + C \cdot \frac{1}{\sum_{t=0}^{T-1} \sum_{i=1}^n \eta^t} \cdot \sum_{t=0}^{T-1} \left(\eta^t \queryRadius^t + (\eta^t)^2 \queryRadius^t + (\eta^t)^2 + \frac{(\eta^t)^2}{\queryRadius^t} + \frac{(\eta^t)^2}{(\queryRadius^t)^2} \right) \\ \label{Eqn: App, bound of w, Ineq, 3}
    \leq \hspace{0.5mm} &\frac{1}{\sum_{t=0}^{T-1} \eta^t} \cdot \frac{3D}{4n} + C \cdot \frac{1}{n\sum_{t=0}^{T-1} \eta^t} \cdot \sum_{t=0}^{T-1} \left(\eta^t \queryRadius^t + (\eta^t)^2 \queryRadius^t + (\eta^t)^2 + \frac{(\eta^t)^2}{\queryRadius^t} + \frac{(\eta^t)^2}{(\queryRadius^t)^2} \right),
\end{align}
By definition, $\eta^t = \eta^0 \cdot (t+1)^{-3/4-\chi}$ and $\queryRadius^t = \queryRadius^0 \cdot (t+1)^{-1/4}$, so by Lemma \ref{Lemma: App, Geometric Sum Bound}, we have:
\begin{align*}
    \sum_{t=0}^{T-1} \eta^t &= \eta^0 \cdot \sum_{t=1}^T t^{-3/4-\chi} \geq 4 \eta^0 \cdot T^{1/4 - \chi}, \\
    \sum_{t=0}^{T-1} \eta^t \queryRadius^t &= \eta^0 \queryRadius^0 \cdot \sum_{t=1}^T t^{-(1+\chi)} \leq \eta^0 \queryRadius^0 \cdot \Bigg( 1 + \frac{1}{\chi} \Bigg), \\
    \sum_{t=0}^{T-1} (\eta^t)^2 &= (\eta^0)^2 \cdot \sum_{t=1}^T t^{-3/2-2\chi} \leq (\eta^0)^2 \cdot \Bigg( 1 + \frac{1}{\frac{1}{2} + 2 \chi} \Bigg) \leq 3 \cdot (\eta^0)^2, \\
    \sum_{t=0}^{T-1} (\eta^t)^2 \queryRadius^t &= (\eta^0)^2 \queryRadius^0 \cdot \sum_{t=1}^T t^{-7/4-2\chi} \leq (\eta^0)^2 \queryRadius^0 \cdot \Bigg( 1 + \frac{1}{\frac{3}{4} + 2 \chi} \Bigg) \leq \frac{7}{4} \cdot (\eta^0)^2 \epsilon^0, \\
    \sum_{t=0}^{T-1} \frac{(\eta^t)^2}{\queryRadius^t} &= \frac{(\eta^0)^2}{\queryRadius^0} \cdot \sum_{t=1}^T t^{-5/4-2\chi} \leq \frac{(\eta^0)^2}{\queryRadius^0} \cdot \Bigg( 1 + \frac{1}{\frac{1}{4} + 2 \chi} \Bigg) \leq 5 \cdot \frac{(\eta^0)^2}{\epsilon^0}, \\
    \sum_{t=0}^{T-1} \frac{(\eta^t)^2}{(\queryRadius^t)^2} &= \frac{(\eta^0)^2}{(\queryRadius^0)^2} \cdot \sum_{t=1}^T t^{-1-2\chi} \leq \frac{(\eta^0)^2}{(\queryRadius^0)^2} \cdot \Bigg( 1 + \frac{1}{2 \chi} \Bigg).
\end{align*}
Substituting back into \eqref{Eqn: App, bound of w, Ineq, 3} and using the convexity of the gap function $\Delta(\cdot)$, we have:
\begin{align*}
    &\E\big[ \Delta(u^T) \big] \\
    \leq \hspace{0.5mm} &\frac{\sum_{t=0}^{T-1} \sum_{i=1}^n \eta^t \cdot \E\big[\Delta(u_i^t) \big]}{\sum_{t=0}^{T-1} \sum_{i=1}^n \eta^t} \\
    \leq \hspace{0.5mm} &\frac{1}{\sum_{t=0}^{T-1} \eta^t} \cdot \frac{3}{4n} D + C \cdot \frac{1}{\sum_{t=0}^{T-1} \eta^t} \cdot \sum_{t=0}^{T-1} \left(\eta^t \queryRadius^t + (\eta^t)^2 \queryRadius^t + (\eta^t)^2 + \frac{(\eta^t)^2}{\queryRadius^t} + \frac{(\eta^t)^2}{(\queryRadius^t)^2} \right) \\
    \leq \hspace{0.5mm} &\Bigg( \frac{3}{16n} D + \frac{47}{4n} \cdot C \max\left\{ \queryRadius^0, \eta^0, \eta^0 \queryRadius^0, \frac{\eta^0}{\queryRadius^0},  \frac{\eta^0}{(\queryRadius^0)^2} \right\} \Big( 1 + \frac{1}{\chi} \Big) \Bigg) T^{-1/4 + \chi} \\
    \leq \hspace{0.5mm} &\queryRadius.
\end{align*}
where the final inequality follows by definition of $T$.
\end{proof}


\newpage

\section{Wasserstein Distributionally Robust Strategic Classification}

\subsection{Model of Adversary}
\label{app: ModelOfAdversary}
In this subsection, we formally define our model for the adversary, and the uncertainty set of distributions for the resulting strategically and adversarially perturbed data. For better exposition, in this section we summarize the various distributions used in the main article in Table \ref{Table: Notation} below.

\begin{table}[h]
    \centering
    \caption{Table of notations}
    \begin{tabular}{||c|c||}
    \hline 
         Notation & Explanation\\ 
         \hline
         \hline
        \(\dataDist\) & Unknown underlying distribution \\
        \(\dataDist(\theta)\) & Unknown underlying distribution strategically perturbed by \(\theta\) \\ 
        \(\td{\dataDist}_n(\theta)\) & Empirical distribution of strategically perturbed data \\ 
        \(\mb{P}\) & An element of uncertainty set \(\mc{P}(\theta)\) \\ 
        \(\mb{P}_\theta^i\) & Conditional distribution of adversarially generated data given \(i^{th}\) data point
        \\
        \hline
        \end{tabular}
    \label{Table: Notation}
\end{table}

The WDRSL problem formulation contains two main components---the \emph{strategic component} that accounts for the distribution shift \(\dataDist(\theta)\) in response to the choice of classifier \(\theta\), and the \emph{adversarial component}, which accounts for the uncertainty set \(\mc{P}(\theta)\). As per the modeling assumptions put forth in Section \ref{ssec: ModelOfResponse}, we have \((\xData_i,\yData_i) \sim \dataDist\) and \((b_i(\theta,\xData_i,\yData_i),\yData_i)\sim \dataDist(\theta)\) for all \(i\in[n]\). 
For the sake of brevity, we shall use \(b_i(\theta)\) in place of \(b_i(\theta,\xData_i,\yData_i)\) for all \(i\in[n]\). 

As per the standard formulation of distributionally robust optimization, we restrict \(\mc{P}(\theta)\) to be a Wasserstein neighborhood of $\td{\dataDist}_n(\theta)$ (the empirical distribution of strategic responses \(\{(b_i(\theta),\yData_i)\}_{i=1}^{n}\)), i.e., we set \(\mc{P}(\theta) \subset \mb{B}_\delta(\td{\dataDist}_n(\theta))\) for some \(\delta>0\). However, to ensure that the min-max problem reformulated from the WDRSC problem is convex-concave, we further require the adversary to modify the \emph{label} of an data point $i$ in the empirical distribution only when the true label $\yData_i$ is \(+1\), although they are still always allowed to modify the \emph{feature} $b_i(\theta)$. As a consequence, this imposes some restrictions on the conditional distribution $\mathbb P_\theta^i$ of $(dx, y)$, as generated by the adversary, given a data point $i$ in the empirical distribution. In particular:
\begin{align*}
    \mathbb{P}^i_\theta(dx, +1|b_i(\theta), -1) = 0, \hspace{5mm} \forall \hspace{0.5mm} i \in [n].
\end{align*}

 By definition of conditional distributions, we obtain that any distribution \(\dist\) can be expressed as the average of the conditional distribution \(\dist_{\theta}^i\). That is, 
\begin{align*}
    \mathbb{P}(dx, y) &= \frac{1}{n} \sum_{i=1}^n  \mathbb{P}^i_\theta(dx, y|b_i(\theta), \tilde{y}_i).
\end{align*}
Below, we formally state the restriction described above.

\begin{assumption}
We assume that \(\dist \in \mb{B}_\delta(\td{\dataDist}_n(\theta))\) and \(\dist^{i}_\theta(dx,+1|b_i(\theta),-1) =0\) for all \(i\in[n]\). As a direct result, the uncertainty set $\mathcal{P}(\theta)$ is characterized as: 
\begin{align}\label{eq: ModelAdversary}
    \pSet(\param)=\mb{B}_{\delta}(\td{\dataDist}_{\numData}(\param)) \cap \left\{ \frac{1}{n} \sum_{i=1}^n \mathbb{P}^i_\theta(dx, y|b_i(\theta), \tilde{y}_i) \Bigg| \mathbb{P}^i_\theta(dx, +1|b_i(\theta), -1) = 0, \hspace{0.5mm} \forall \hspace{0.5mm} i \in [n] \right\}.
\end{align}
\end{assumption}


In the following subsection, we reformulate the WDRSC problem with a generalized linear model and with the uncertainty set defined in \eqref{eq: ModelAdversary}.

\subsection{Proof of Theorem \ref{thm: WDRSC-convex-concave}}
\label{app: appendixProofofConvexConcavity}

The proof takes inspirations from  \cite[Theorem 1]{ShafieezadehAbadeh2015DistributionallyRL}. First, we define the \textit{Wasserstein distance between distributions on $\mathcal{Z}$ with cost function $c$}; note that this is different from the \textit{$p$-Wasserstein distance between probability distributions on $\R^d$} defined in Appendix \ref{subsec: App, Lemmas for OGDA Theorem}.

\begin{definition}(\textbf{Wasserstein distance between distributions on $\mathcal{Z}$ with cost Function $c$}) \label{Def: Wasserstein Metric, with Cost}
Let $\mu, \nu$ be probability distributions over $\mathcal{Z} := \R^d \times \{+1, -1\}$ with finite second moments, and let $\Pi(\mu, \nu)$ denote the set of all couplings (joint distributions) between $\mu$ and $\nu$. Given a metric $c: \mathcal{Z} \times \mathcal{Z} \ra [0, \infty)$ on $\mathcal{Z}$, we define:
\begin{align*}
    \W_c(\mu, \nu) = \inf_{(Z, Z') \sim \pi \in \Pi(\mu, \nu)} \E_\pi \big[ c(Z, Z') \big].
\end{align*}
\end{definition}

In Theorem \ref{Eqn: WDRSC Problem} and in our proof below, we use the cost function $c(z, z') := \Vert x - x' \Vert_2^2 + \kappa \cdot |y - y'|$, with a fixed constant $\kappa > 0$, for each $z := (x, y) \in \mathcal{Z}$ and $z' := (x', y') \in \mathcal{Z}$.

\begin{proof}(\textbf{Proof of Theorem \ref{thm: WDRSC-convex-concave}})
Fix a \(\theta\in\Theta\). Note that \(b_i(\theta,\xData_i,+1) = \xData_i\). For any \((x,y)\in \R^d\times \{-1,1\}\), let \(\ell((x,y),\theta) \defas \phi(\lara{x,\theta}) - y\lara{x,\theta}\).   We first analyze the inner supremum term, i.e. 
\begin{align*}
    &\hspace{5mm} \sup_{\dist\in \mc{P}(\theta)} \mb{E}_{\dist}[\phi(\langle x,\theta\rangle)-y\langle x,\theta\rangle] \\
    &=  \sup_{\dist\in \mc{P}(\theta)}\int_{\mc{Z}}\ell(z,\theta)\dist(z) dz \\
   &=\left\{\begin{array}{cc}
         \sup\limits_{\pi_\theta\in \Pi(\dist,\td{\dataDist}_{n}(\theta))} \ & \int_{\mc{Z}}\ell(z,\theta)\pi_\theta(dz,\mc{Z})  \\
        \text{s.t.}\ &\ \int_{\mc{Z}\times \mc{Z}}\|z-\td{z}\|\pi_\theta(dz,d\td{z})\leq \delta
    \end{array}\right.
\end{align*}
Here, $\Pi(\mathbb{P}, \tilde{\dataDist}_n(\theta))$ denotes the set of all joint distributions that couple $\mathbb{P} \in \mathcal{P}(\theta)$ and $\tilde{\dataDist}_n(\theta)$.
Since the marginal distribution $\td{\dataDist}_{n}(\theta)$ of $\td{z}$ is discrete, such couplings $\pi_\theta$ are completely determined by the conditional distribution $\dist^i_{\theta}$ of $z$ given $\td{z}_i=(\xData_i(\theta),\yData_i)$ for each $i\in \{1,\ldots, n\}$. That is: 
\[
\pi_\theta(dz,d\td{z}) = \frac{1}{n}\sum_{i\in[n]} \vartheta_{(b_i(\theta),\yData_i)}(d\td{z}) \dist_{\theta}^i(dz)
\]
where for any \((x,y)\in\mc{Z}\), \(\vartheta_{(x,y)}\) is a \emph{Dirac delta} distribution with its support at point \((x,y)\). 

We introduce some notations. Let \(\mc{I}_{+1} = \{i\in[n]:\yData_{i}=+1\}\) and \(\mc{I}_{-1} = \{i\in[n]:\yData_{i}=-1\}\). Let's introduce two distributions \(\mu_{\theta}^i\) and \(\nu_{\theta}^{i}\) such that 
\[
\dist_{\theta}^i = \begin{cases}
\mu_{\theta}^{i} & \text{if} \ i \in \mc{I}_{+1} \\ 
\nu_{\theta}^{i} & \text{if}  \ i \in \mc{I}_{-1}
\end{cases}
\]
Due to the constraint \eqref{eq: ModelAdversary}, we have \(\nu_\theta^i(dx,+1)=0\) at every \(x\). This implies:
\[\pi_\theta(dz,d\td{z})=\frac{1}{n}\lr{\sum_{i\in\mc{I}_{+1}}\vartheta_{(b_i(\theta),1)}(d\td{z})\mu^i_\theta(dz) + \sum_{i\in\mc{I}_{-1}}\vartheta_{(b_i(\theta),-1)}(d\td{z})  \nu_{\theta}^i(dz)}
\]
With a slight abuse of notation, we denote \(\mu_{\theta,+1}^i(dx) = \mu_{\theta}^i(dx,+1)\), \(\mu_{\theta,-1}^i(dx) = \mu_{\theta}^i(dx,-1)\) and \(\nu_\theta^i(dx) = \nu_\theta^i(dx,-1)\). 
The optimization problem of concern then simplifies to:
\begin{align*}
    \sup_{\mu_{\theta,\pm 1}^i,\nu_{\theta}^i}\ &\ \frac{1}{n}\sum_{i\in\mc{I}_{+1} } \int_{\mb{R}^d} \ell((x,+1),\theta)\mu_{\theta,+1}^i(dx)+\frac{1}{n}\sum_{i\in\mc{I}_{+1}} \int_{\mb{R}^d} \ell((x,-1),\theta)\mu_{\theta,-1}^i(dx) \\
    &\hspace{1cm} +  \frac{1}{n}\sum_{i\in\mc{I}_{-1}}\int_{\mb{R}^d} \ell((x,-1),\theta)\nu_{\theta}^i(dx) \\
    \text{s.t.}\ & \ \frac{1}{n}\sum_{i:\yData_i=+1}\int_{\mb{R}^d}\|(x,+1)-(b_i(\theta),\yData_i)\|\mu_{\theta,+1}^i(dx)\\
    &\hspace{1cm} +\frac{1}{n}\sum_{i:\yData_i=+1}\int_{\mb{R}^d}\|(x,-1)-(b_i(\theta),\yData_i)\|\mu_{\theta,-1}^i(dx) \\
  &\ \int_{\mb{R}^d}\mu_{\theta,+1}^i(dx)+\int_{\mb{R}^d}\mu_{\theta,-1}^i(dx)=1, \quad \forall \quad i \in \mc{I}_{+1} \\ 
    &\ \int_{\R^d}\nu_\theta^i(dx)=1 , \quad \forall \quad i\in\mc{I}_{-1}
\end{align*}
First, we rewrite the inequality constraint above as follows. Recall that:
\begin{align*}
    &\frac{2\kappa}{n}\int_{\mb{R}^d}\sum_{i\in\mc{I}_{+1}}\mu_{\theta,-1}^i(dx) + \frac{1}{n}\int_{\mb{R}^d}\sum_{i\in\mc{I}_{+1}} \|x-b_i(\theta)\|\mu_{\theta,+1}^i(dx) \\
    &\hspace{1cm} + \frac{1}{n}\int_{\mb{R}^d}\sum_{i\in\mc{I}_{+1}}  \|x-b_i(\theta)\|\mu_{\theta,-1}^i(dx) + \frac{1}{n}\int_{\mb{R}^d}\sum_{i\in\mc{I}_{-1}} \|x-b_i(\theta)\|\nu_{\theta}^i (dx)\leq \delta.
\end{align*}
Hence,
\begin{align*}
    \sup_{\mu_{\theta, \pm 1}^i,\nu_{\theta}^i}\ &\ \frac{1}{n}\sum_{i\in\mc{I}_{+1}} \int_{\mb{R}^d} \ell((x,+1),\theta)\mu_{\theta,+1}^i(dx)+\frac{1}{n}\sum_{i\in\mc{I}_{+1}} \int_{\mb{R}^d} \ell((x,-1),\theta)\mu_{\theta,-1}^i(dx) \\
    &\hspace{1cm} +
     \frac{1}{n}\sum_{\yData_i=-1}\int_{\mb{R}^d} \ell((x,-1),\theta)\nu_{\theta}^i(dx)\\
    \text{s.t.}\ &\frac{2\kappa}{n}\int_{\mb{R}^d}\sum_{i\in\mc{I}_{+1}}\mu_{\theta,-1}^i(dx) + \frac{1}{n}\int_{\mb{R}^d}\sum_{i\in\mc{I}_{+1}} \|x-b_i(\theta)\|\mu_{\theta,+1}^i(dx) \\
    &\hspace{1cm} + \frac{1}{n}\int_{\mb{R}^d}\sum_{i\in\mc{I}_{+1}}  \|x-b_i(\theta)\|\mu_{\theta,-1}^i(dx) + \frac{1}{n}\int_{\mb{R}^d}\sum_{i\in\mc{I}_{-1}} \|x-b_i(\theta)\|\nu_{\theta}^i (dx)\leq \delta\\
    &\ \int_{\mb{R}^d}\mu_{\theta,+1}^i(dx)+\int_{\mb{R}^d}\mu_{\theta,-1}^i(dx)=1, \quad \forall \quad i\in\mc{I}_{+1} \\ 
    &\ \int_{\R^d}\nu_\theta^i(dx)=1 , \quad \forall \quad i\in\mc{I}_{-1}
\end{align*}


Now, we can use duality to reformulate the infinite-dimensional optimization problem into a finite-dimensional problem:
\begin{align*}
    &\sup_{\dist\in \mc{P}(\theta)} \mb{E}_{\dist}[\phi(\langle x,\theta\rangle)-y\langle x,\theta\rangle] \\
    = \hspace{0.5mm} &\begin{cases}
    \inf_{\alpha,s_i} &\alpha \delta +\frac{1}{n}\sum_{i\in\mc{I}_{+1}}s_i+\frac{1}{n}\sum_{i\in\mc{I}_{-1}}t_i \\ 
    \text{s.t.} & \sup_{x}  \ell((x,+1),\theta) -\alpha \cdot \frac{1+\yData_i}{2}\|x-b_i(\theta)\|\leq s_i \quad \forall \ i\in\mc{I}_{+1}\\
    & \sup_{x}  \ell((x,-1),\theta) -\alpha \cdot \frac{1+\yData_i}{2}\|x-b_i(\theta)\|-\alpha\kappa (1+\yData_i)\leq s_i \quad \forall \ i\in\mc{I}_{+1}\\
    &\sup_{x}  \ell((x,-1),\theta) -\alpha \cdot \frac{1-\yData_i}{2}\|x-b_i(\theta)\|  \leq t_i \quad \forall \ i\in\mc{I}_{-1}\\
    &\alpha \geq 0
    \end{cases}
\end{align*}
which is equivalent to: 
\begin{align*}
    &\sup_{\dist\in \mc{P}(\theta)} \mb{E}_{\dist}[\phi(\langle x,\theta\rangle)-y\langle x,\theta\rangle] \\
    = \hspace{0.5mm} &\begin{cases}
    \inf_{\alpha,s_i} &\alpha \delta +\frac{1}{n}\sum_{i\in\mc{I}_{+1}}s_i+\frac{1}{n}\sum_{i\in\mc{I}_{-1}}t_i \\ 
    \text{s.t.} & \sup_{x}  \ell((x,+1),\theta) -{\alpha}\|x-b_i(\theta)\|\leq s_i \quad \forall \ i\in\mc{I}_{+1}\\
    & \sup_{x}  \ell((x,-1),\theta) -{\alpha}\|x-b_i(\theta)\|-2\alpha\kappa \leq s_i \quad \forall \ i\in\mc{I}_{+1}\\
    &\sup_{x}  \ell((x,-1),\theta) -{\alpha}\|x-b_i(\theta)\|  \leq t_i \quad \forall \ i\in\mc{I}_{-1} \\
    &\alpha \geq 0
    \end{cases}
\end{align*}
We now invoke \cite[Lemma A.1]{YuMazumdar2021FastDistributionallyRobustLearning}, which claims that for any \(\td{y}\in \{+1,-1\}\) and \(\td{x}\in \R^d\):
\begin{align*}
    \sup_x \ell((x,\td{y}),\theta) - \alpha \|x-\td{x}\|  = \begin{cases}
    \ell((\td{x},\td{y}),\theta) & \ \text{if} \ \|\theta\| \leq \alpha/(L+1) \\ 
    -\infty  &  \ \text{otherwise.}
    \end{cases}
\end{align*}

We now have:
\begin{align*}
    &\sup_{\dist\in \mc{P}(\theta)} \mb{E}_{\dist}[\phi(\langle x,\theta\rangle)-y\langle x,\theta\rangle] \\
    = \hspace{0.5mm} &\begin{cases}
    \inf_{\alpha,s_i} &\alpha \delta +\frac{1}{n}\sum_{i\in\mc{I}_{+1}}s_i+\frac{1}{n}\sum_{i\in\mc{I}_{-1}}t_i \\ 
    \text{s.t.} & 
    \ell((b_i(\theta),+1),\theta)\leq s_i \quad \forall \ i\in\mc{I}_{+1}\\
    & \ell((b_i(\theta),-1),\theta) -2\alpha\kappa \leq s_i \quad \forall \ i\in\mc{I}_{+1}\\
    &\ell((b_i(\theta),-1),\theta)  \leq t_i \quad \forall \ i\in\mc{I}_{-1}\\
    &\alpha \geq 0 \\ 
    &\|\theta\|\leq \alpha/(L+1)    
    \end{cases}
\end{align*}

In the above presented optimization problem we can conclude that:
\begin{align*}
    t_i &= \phi(\lara{b_i(\theta),\theta})+\lara{b_i(\theta),\theta} && \forall i\in\mc{I}_{-1} \\ s_i &= \max\{\ell((b_i(\theta),+1),\theta),  \ell((b_i(\theta),-1),\theta) -2\alpha\kappa\} && \forall i\in\mc{I}_{+1}.
\end{align*}
To further simplify the \(s_i\) expression, note that:
\begin{align*}
    s_i &= \max\{ \phi(\lara{b_i(\theta),\theta})- \lara{b_i(\theta),\theta},\phi(\lara{b_i(\theta),\theta})+ \lara{b_i(\theta),\theta}-2\alpha\kappa  \} \\
    &= \phi(\lara{b_i(\theta),\theta})- \lara{b_i(\theta),\theta} + \max\{ 0,2 \lara{b_i(\theta),\theta}-2\alpha\kappa  \} \\ 
    &= \phi(\lara{b_i(\theta),\theta})-\alpha\kappa+ \max_{\gamma_i:|\gamma_i|\leq 1}\gamma_i\lr{ \lara{b_i(\theta),\theta}-\alpha\kappa},
\end{align*}
so the overall objective can be written as:
\begin{align*}
    &\sup_{\dist\in \mc{P}(\theta)} \mb{E}_{\dist}[\phi(\langle x,\theta\rangle)-y\langle x,\theta\rangle] \\
    = \hspace{0.5mm} &\begin{cases}
    \inf_{\alpha} \max_{\gamma:\|\gamma\|_{\infty}\leq 1} &\alpha (\delta-\kappa) +\frac{1}{n}\sum_{i}\frac{1+\yData_i}{2}\lr{\phi(\lara{b_i(\theta),\theta})+\gamma_i \lr{\lara{b_i(\theta),\theta}-\alpha\kappa} }  \\
    &\hspace{1cm} +\frac{1}{n}\sum_i\frac{1-\yData_i}{2}\lr{ \phi(\lara{b_i(\theta),\theta})+\lara{b_i(\theta),\theta}} \\ 
    \text{s.t.} 
    &\|\theta\|\leq \alpha/(L+1)    
    \end{cases}
\end{align*}

We claim that the minimax objective above is convex is \(\theta\). There are mainly two cases to analyze:
\begin{enumerate}
    \item \textbf{Case I (\(i\in\mc{I}_{+1}\))}:  We have \(b_i(\theta) = \xData_i\) as per the strategic classification model. Therefore \(\lara{b_i(\theta),\theta}\) is a linear function.  For every \(\gamma,\alpha\), we claim that the mapping \(\theta \mapsto \phi(\lara{b_i(\theta),\theta})+\gamma_i (\lara{b_i(\theta),\theta}-\alpha\kappa)\) is convex. Indeed, the assumption that \(\phi\) is convex and the observation that \(\lara{b_i(\theta),\theta}\) is affine in \(\theta\) ensures the convexity. 
    \item \textbf{Case II (\(i\in\mc{I}_{-1}\))}: We know from Lemma \ref{lem: ConvexityBR} that \(\lara{b_i(\theta),\theta}\) is convex in \(\theta\). Moreover, the convexity of \(\phi\) and the assumption that \(z\mapsto \phi(z)+z\) is non-decreasing ensures that \(\phi(\lara{b_i(\theta),\theta})+\lara{b_i(\theta),\theta}\) is convex for every \(i\). 
\end{enumerate}
This concludes the proof.

\end{proof}

\newpage

\section{Details on experimental study}
\label{sec: App, ExpStudy}

{Code used to reproduce the results in the main paper is available at \url{https://drive.google.com/drive/folders/1spuB3R6vEU2AqaXxAxeeXo9z5QMVdtdl?usp=sharing}}


\subsection{Algorithms}
\label{subsec: App, AlgoExp}

In our experiments, we compare the OGDA-RR algorithm (Alg. \ref{Alg: OGDA-RR}) with three other zeroth-order algorithms---Optimistic Gradient Descent Ascent with Sampling with Replacement (OGDA-WR), Stochastic Gradient Descent Ascent with Random Reshuffling (SGDA-RR), and Stochastic Gradient Descent Ascent with Sampling with Replacement (SGDA-WR)---characterized by the update equations \eqref{Eqn: OGDA-WR, Update}, \eqref{Eqn: SGDA-RR, Update}, \eqref{Eqn: SGDA-WR, Update}, respectively. For convenience, we have reproduced \eqref{Eqn: OGDA-RR, Update}, the update equation for the OGDA-RR algorithm (Algorithm \ref{Alg: OGDA-RR}), as \eqref{Eqn: OGDA-RR, Update, Copy} below:
\begin{align} \label{Eqn: OGDA-RR, Update, Copy}
    u_{i+1}^t &= \proj_{\X \times \Y} \Big(u_{i}^t - \eta^t \hat{F}_{\sigma_i^t}(u_{i}^t; \queryRadius^t,v_{i}^t)-\eta^t \hat{F}_{\sigma_{i-1}^t}(u_{i}^t;\queryRadius^t,v_{i}^t) + \eta^t\hat{F}_{\sigma_{i-1}^t}(u_{i-1}^t;\queryRadius^t,v_{i-1}^t) \Big), \\
    \label{Eqn: OGDA-WR, Update}
    u_{i+1}^t &= \proj_{\X \times \Y} \Big(u_{i}^t - \eta^t \hat{F}_{j_i^t}(u_{i}^t; \queryRadius^t,v_{i}^t)-\eta^t \hat{F}_{j_{i-1}^t}(u_{i}^t;\queryRadius^t,v_{i}^t) + \eta^t\hat{F}_{j_{i-1}^t}(u_{i-1}^t;\queryRadius^t,v_{i-1}^t) \Big), \\ \label{Eqn: SGDA-RR, Update}
    u_{i+1}^t &= \proj_{\X \times \Y} \Big(u_{i}^t - \eta^t \hat{F}_{\sigma_i^t}(u_{i}^t; \queryRadius^t,v_{i}^t) \Big), \\ \label{Eqn: SGDA-WR, Update}
    u_{i+1}^t &= \proj_{\X \times \Y} \Big(u_{i}^t - \eta^t \hat{F}_{j_i^t}(u_{i}^t; \queryRadius^t,v_{i}^t) \Big),
\end{align}
where the indices $\sigma_i^t$ and $j_i^t$ are as defined in Algorithms \ref{Alg: OGDA-WR}, \ref{Alg: SGDA-RR}, and \ref{Alg: SGDA-WR}.

\begin{algorithm}

\SetAlgoLined

 \textbf{Input}: stepsizes $\eta^t,\queryRadius^t$, data points \(\{(x_i,y_i)\}_{i=1}^{n}\sim\mathcal{D},u_0^{(0)}\), time horizon duration $T$;
 
 \For{$t=0,1, \cdots, T-1$}{

\For{$i=0,\ldots, n-1$}{

Sample $j_i^t \sim \Unif(\{1, \cdots, n\})$

Sample $v_i^t \sim  \Unif(\sphere^{d-1})$

$u_{i+1}^t \gets \eqref{Eqn: OGDA-WR, Update}$

}

$u_{0}^{(t+1)} \gets u_{n}^t$

$u_{-1}^{(t+1)} \gets u_{n-1}^t$

 }
 
\textbf{Output:} $\tilde u^T := \frac{1}{n \cdot \sum_{t=0}^{T-1} \eta^t} \sum_{t=0}^{T-1} \sum_{i=1}^n \eta^t u_i^t$.
 
\caption{OGDA-WR Algorithm}
\label{Alg: OGDA-WR}
\end{algorithm}

\begin{algorithm}

\SetAlgoLined

 \textbf{Input}: stepsizes $\eta^t,\queryRadius^t$, data points \(\{(x_i,y_i)\}_{i=1}^{n}\sim\mathcal{D},u_0^{(0)}\), time horizon duration $T$;
 
 \For{$t=0,1, \cdots, T-1$}{

\For{$i=0,\ldots, n-1$}{

Sample $j_i^t \sim \Unif(\{1, \cdots, n\})$

Sample $v_i^t \sim  \Unif(\sphere^{d-1})$

$u_{i+1}^t \gets \eqref{Eqn: SGDA-RR, Update}$


}

$u_{0}^{(t+1)} \gets u_{n}^t$

$u_{-1}^{(t+1)} \gets u_{n-1}^t$

 }
 
\textbf{Output:} $\tilde u^T := \frac{1}{n \cdot \sum_{t=0}^{T-1} \eta^t} \sum_{t=0}^{T-1} \sum_{i=1}^n \eta^t u_i^t$.
 
\caption{SGDA-RR Algorithm}
\label{Alg: SGDA-RR}
\end{algorithm}

\begin{algorithm}

\SetAlgoLined

 \textbf{Input}: stepsizes $\eta^t,\queryRadius^t$, data points \(\{(x_i,y_i)\}_{i=1}^{n}\sim\mathcal{D},u_0^{(0)}\), time horizon duration $T$;
 
 \For{$t=0,1, \cdots, T-1$}{
$\sigma^t = (\sigma_1^t, \cdots, \sigma_n^t) \gets$ a random permutation of set \([n]\)\;

\For{$i=0,\ldots, n-1$}{
Sample $v_i^t \sim  \Unif(\sphere^{d-1})$

$u_{i+1}^t \gets \eqref{Eqn: SGDA-WR, Update}$

}

$u_{0}^{(t+1)} \gets u_{n}^t$

$u_{-1}^{(t+1)} \gets u_{n-1}^t$

 }
 
\textbf{Output:} $\tilde u^T := \frac{1}{n \cdot \sum_{t=0}^{T-1} \eta^t} \sum_{t=0}^{T-1} \sum_{i=1}^n \eta^t u_i^t$.
 
\caption{SGDA-WR Algorithm}
\label{Alg: SGDA-WR}
\end{algorithm}

\subsection{Additional Experimental Results}
\label{subsec: Additional Experimental Results}

In this section, we present more experimental findings, on both synthetic and real-world datasets, that reinforces the utility of the proposed algorithm. In all experimental results throughout this subsection, we take \(\delta=0.4\), \(\kappa = 0.5\) and \(\zeta = 0.05\).

\subsubsection{Experimental Study On Synthetic Datasets}
\label{subsubsec: Experimental study on synthetic datasets}

Figure \ref{fig:4000 Synthetic} compares the performance of \ref{OGDARR}-\ref{SGDAS} on a synthetic dataset (whose generating process is the same as that described in Section \ref{sec: ExpResults}), with 4000 training points and 800 test points. Our proposed algorithm performs better empirically compared to most of its counterparts. Moreover, the proposed classifier, \ref{OGDARR}, is significantly more robust than a classifier obtained without considering adversarial perturbations. Note, however, that we cannot make any conclusive claims yet, because of the inherent randomness in these algorithms. Indeed, even if we fix the initialization, then there are two sources of randomness---the construction of the zeroth-order gradient estimator, and the sampling process that generates the data points.

To illustrate the variability in these algorithms' performance, we run each algorithm repeatedly on a data set with 500 synthetically generated data points, using the same initialization, and present confidence interval plots with $\pm 2$ standard deviations for the resulting performance (Figure \ref{fig:error bars}). On average, our proposed algorithm \ref{OGDARR}  outperforms the other algorithms \ref{OGDAS}-\ref{SGDAS}. It is also interesting to point out that the performance of algorithms with random reshuffling is generally higher, and fluctuate less, compared to the performance of algorithms without random reshuffling. 

\begin{figure}
    \centering
    \includegraphics[scale=0.6]{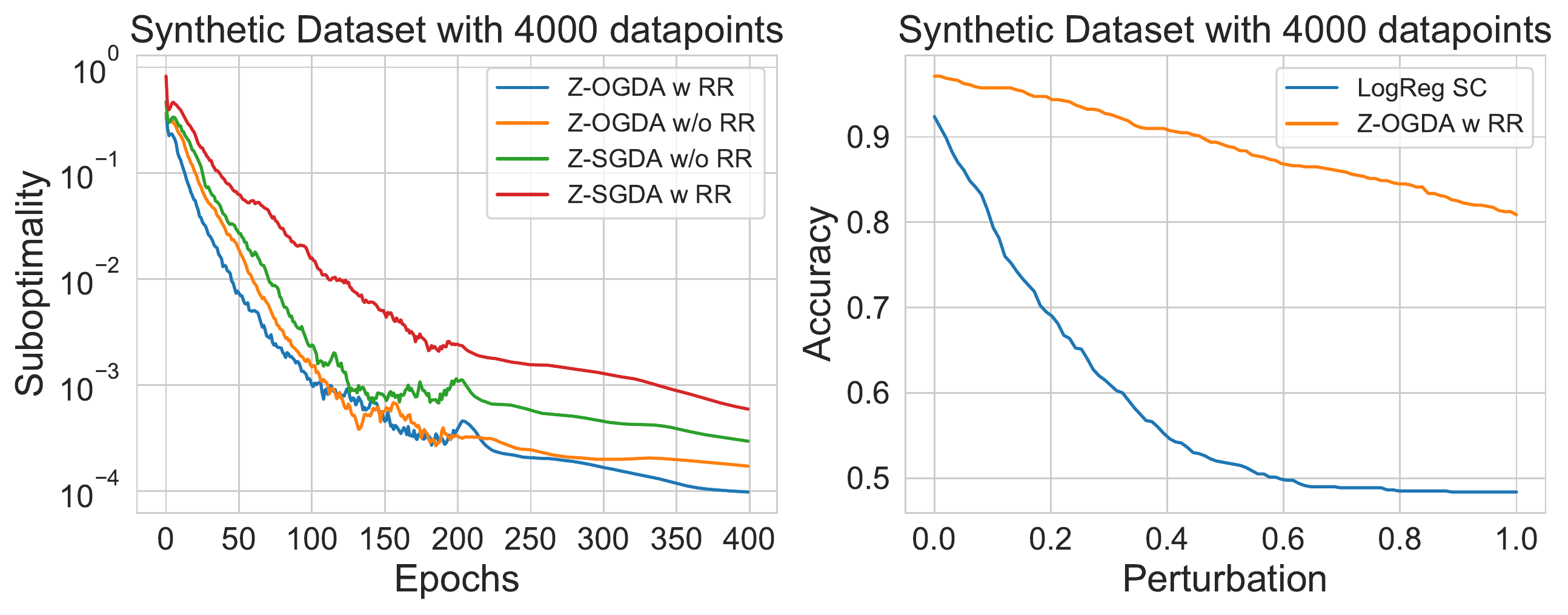}
    \caption{Experimental results for a synthetic dataset with \(n=4000\). (Left pane)) Suboptimality iterates generated by the four algorithms \ref{OGDARR}, \ref{OGDAS}, \ref{SGDARR}, \ref{SGDAS}, respectively denoted as \emph{Z-OGDA w RR}, \emph{Z-OGDA w/o RR}, \emph{Z-SGDA w RR}, \emph{Z-SGDA w/o RR}. (Right pane ) Comparison between decay in accuracy of strategic classification with logistic regression (trained with \(\zeta=0.05\)) and Alg. \ref{OGDARR} with changes in perturbation.}
    \label{fig:4000 Synthetic}
\end{figure}

\begin{figure}
    \centering
    \includegraphics[scale=0.5]{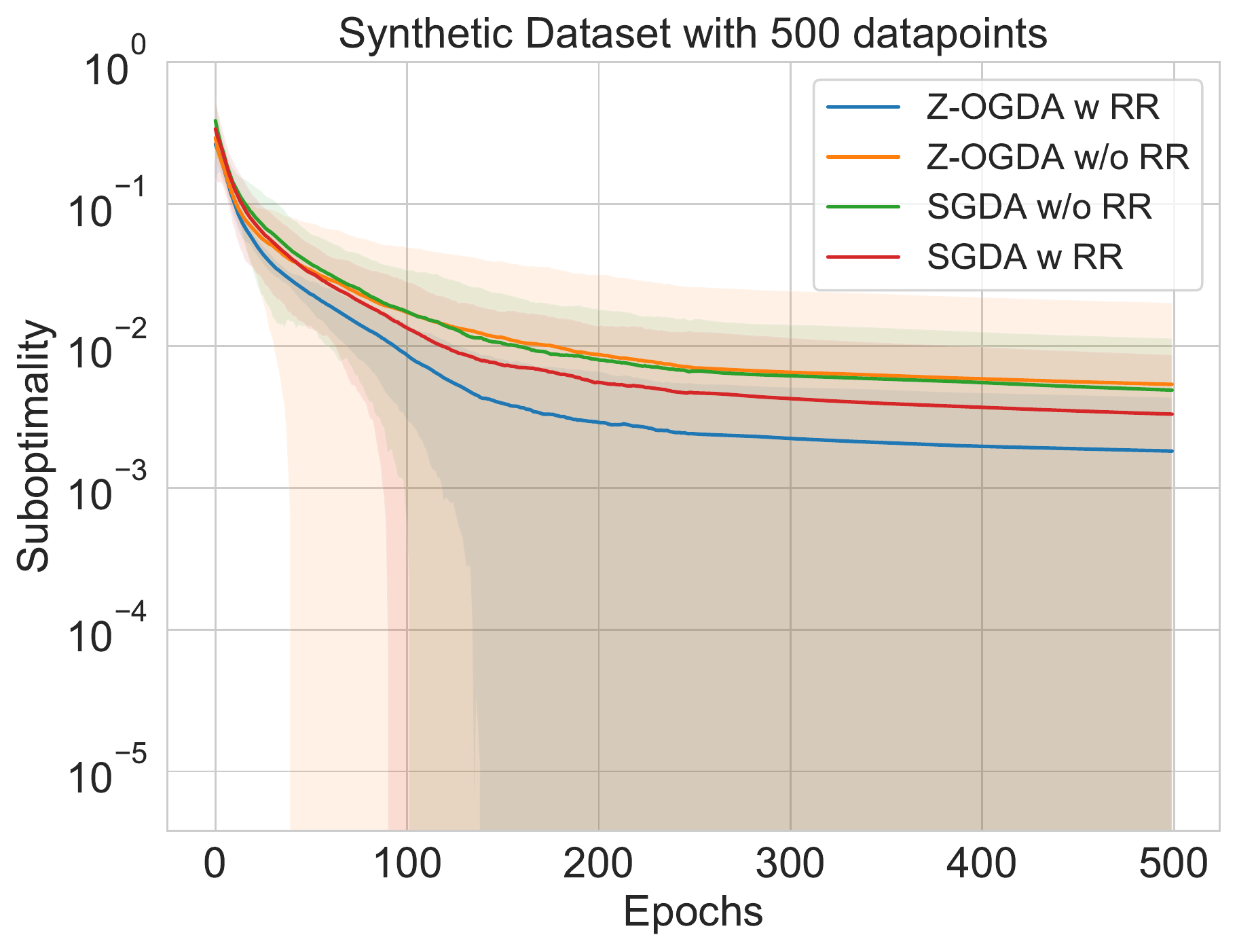}
    \caption{Experimental results for a synthetic dataset with \(n=500\).  Suboptimality iterates generated by the four algorithms \ref{OGDARR}, \ref{OGDAS}, \ref{SGDARR}, and \ref{SGDAS} are respectively denoted as \emph{Z-OGDA w RR}, \emph{Z-OGDA w/o RR}, \emph{Z-SGDA w RR}, and \emph{Z-SGDA w/o RR}.}
    \label{fig:error bars}
\end{figure}

We now illustrate the performance of our algorithm on two real-world data sets---the \say{GiveMeSomeCredit} dataset \footnote{This dataset can be found at \url{https://www.kaggle.com/c/GiveMeSomeCredit}}, and the \say{Porto Bank} data set\footnote{This dataset can be found at \url{https://archive.ics.uci.edu/ml/datasets/bank+marketing}}.

\subsubsection{Experimental Study on Credit Dataset}
\label{subsubsec: Experimental Study on Credit Dataset}

In modern times, banks use machine learning to determine whether or not to finance a customer. This process can be encoded into a classification framework, by using features such as age, debt ratio, monthly income to classify a customer as either likely or unlikely to default. However, those algorithms generally do not account for strategic or adversarial behavior on the part of the agents. 

To illustrate the effect of our algorithm on datasets of practical significance, we deploy our algorithms on the \say{GiveMeSomeCredit}(GMSC) dataset, while assuming that the underlying features are subject to strategic or adversarial perturbations. We use a subset of the dataset of size 2000 with balanced labels. In Figure \ref{fig:GMSC}, we compare the empirical performance of our algorithm \ref{OGDARR} with that of \ref{OGDAS}-\ref{SGDAS}. The left pane shows that \ref{OGDARR} performs well, and the right pane illustrates that our classifier is significantly more robust to adversarial perturbations in data, compared to the strategic classification-based logistic regression algorithm developed recently in the literature \cite{Dong2018StrategicClassification}.

\begin{figure}
    \centering
    \includegraphics[scale=0.6]{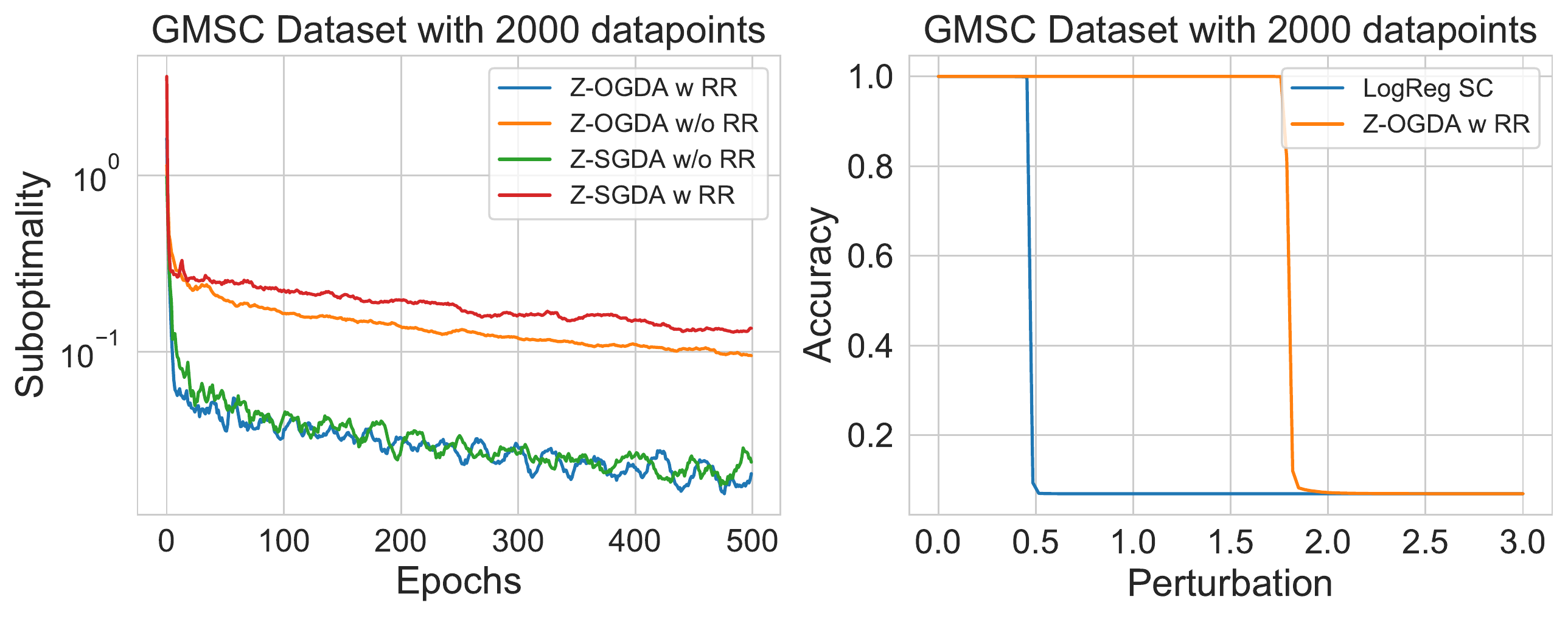}
    \caption{Experimental results for a balanced GiveMeSomeCredit dataset with \(n=2000\). (Left pane) Suboptimality iterates generated by the four algorithms \ref{OGDARR}, \ref{OGDAS}, \ref{SGDARR}, \ref{SGDAS}, respectively denoted as \emph{Z-OGDA w RR}, \emph{Z-OGDA w/o RR}, \emph{Z-SGDA w RR}, \emph{Z-SGDA w/o RR}. (Right pane) Comparison between decay in accuracy of strategic classification with logistic regression (originally trained with \(\zeta=0.05\)) and Alg. \ref{OGDARR} with changes in perturbation.}
    \label{fig:GMSC}
\end{figure}

\subsubsection{Experimental Study on Porto-Bank Dataset}
\label{subsubsec: Experimental Study on Porto-Bank Dataset}

Next, we present empirical results obtained by applying our algorithm to the \say{Porto-Bank} dataset, which describes marketing campaigns of term deposits at Portuguese financial institutions. The classification task in this scenario aims to predict whether a customer with given features (eg. age, job, marital status etc.) would enroll for term deposits. 

In Figure \ref{fig:Porto}, we present the performance of our proposed algorithm \ref{OGDARR} on the Porto-Bank dataset. For ease of illustration, we consider a subset of the dataset with 2000 training data points, 800 test data points, and balanced labels. In Figure \ref{fig:Porto}, we compare the empirical performance of our algorithm \ref{OGDARR} with that of \ref{OGDAS}-\ref{SGDAS}. The left pane shows that \ref{OGDARR} performs well, while the right pane illustrates that our classifier is significantly more robust to adversarial perturbations in data, compared to the strategic classification-based logistic regression developed recently in the literature \cite{Dong2018StrategicClassification}.

\begin{figure}
    \centering
    \includegraphics[scale=0.6]{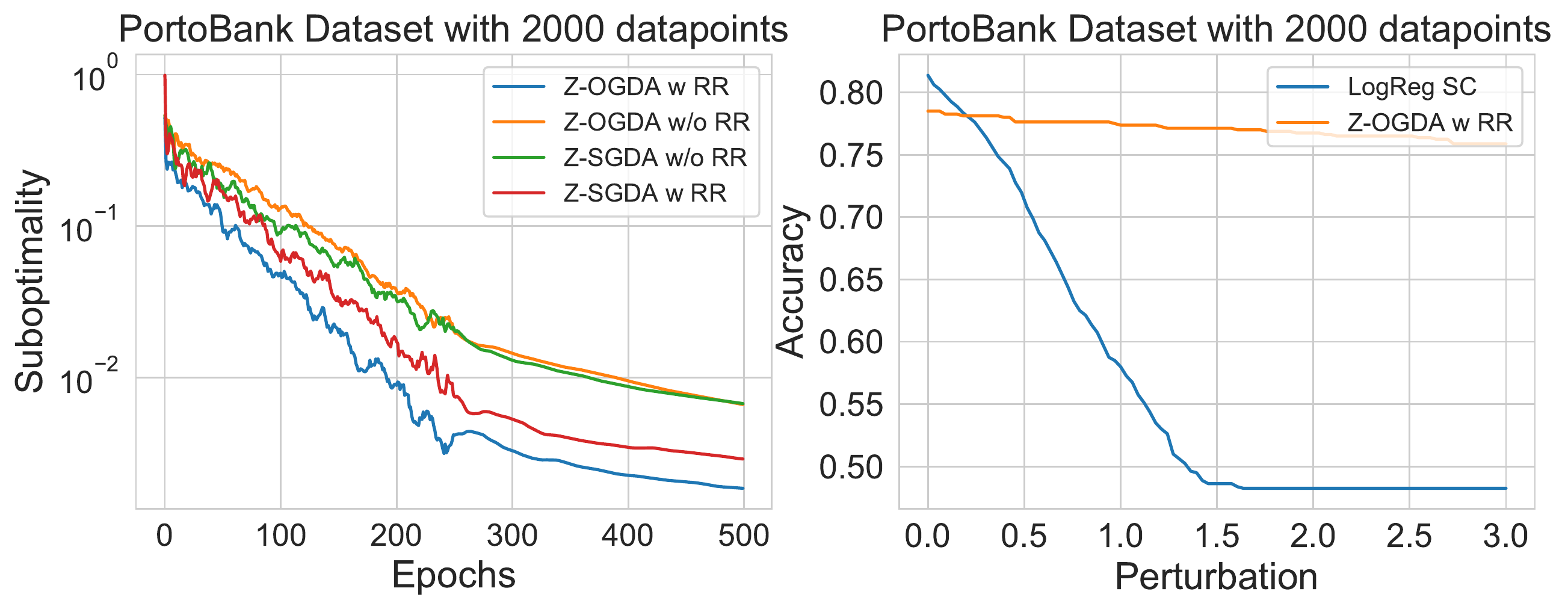}
    \caption{Experimental results for a balanced PortoBank dataset with \(n=2000\). (Left pane) Suboptimality iterates generated by the four algorithms \ref{OGDARR}, \ref{OGDAS}, \ref{SGDARR}, \ref{SGDAS}, respectively denoted as \emph{Z-OGDA w RR}, \emph{Z-OGDA w/o RR}, \emph{Z-SGDA w RR}, \emph{Z-SGDA w/o RR}. (Right pane) Comparison between decay in accuracy of strategic classification with logistic regression (originally trained with \(\zeta=0.05\)) and Alg. \ref{OGDARR} with change in perturbation.}
    \label{fig:Porto}
\end{figure}
\subsubsection{Effect of \(n,d\) on sample complexity}
In this part, we demonstrate the empirical results that corroborates the theoretical dependence of sample complexity on \(n,d\). For this purpose, we use synthetic dataset which is generated as per the method described in Section 6.1. Here we work in the setting where \(n\in \{500,1000,1500,2000\}\) and \(d\in \{10,15,20,25\}\). We fix the suboptimality to \(\epsilon = 0.1\) and compute the number of samples required in each of the settings of \(n\) and \(d\) so that the iterates reach the \(\epsilon-\)suboptimality. We present the results in Figure \ref{fig:DependenceOnN}.  

\begin{figure}
    \centering
    \includegraphics[scale=0.36]{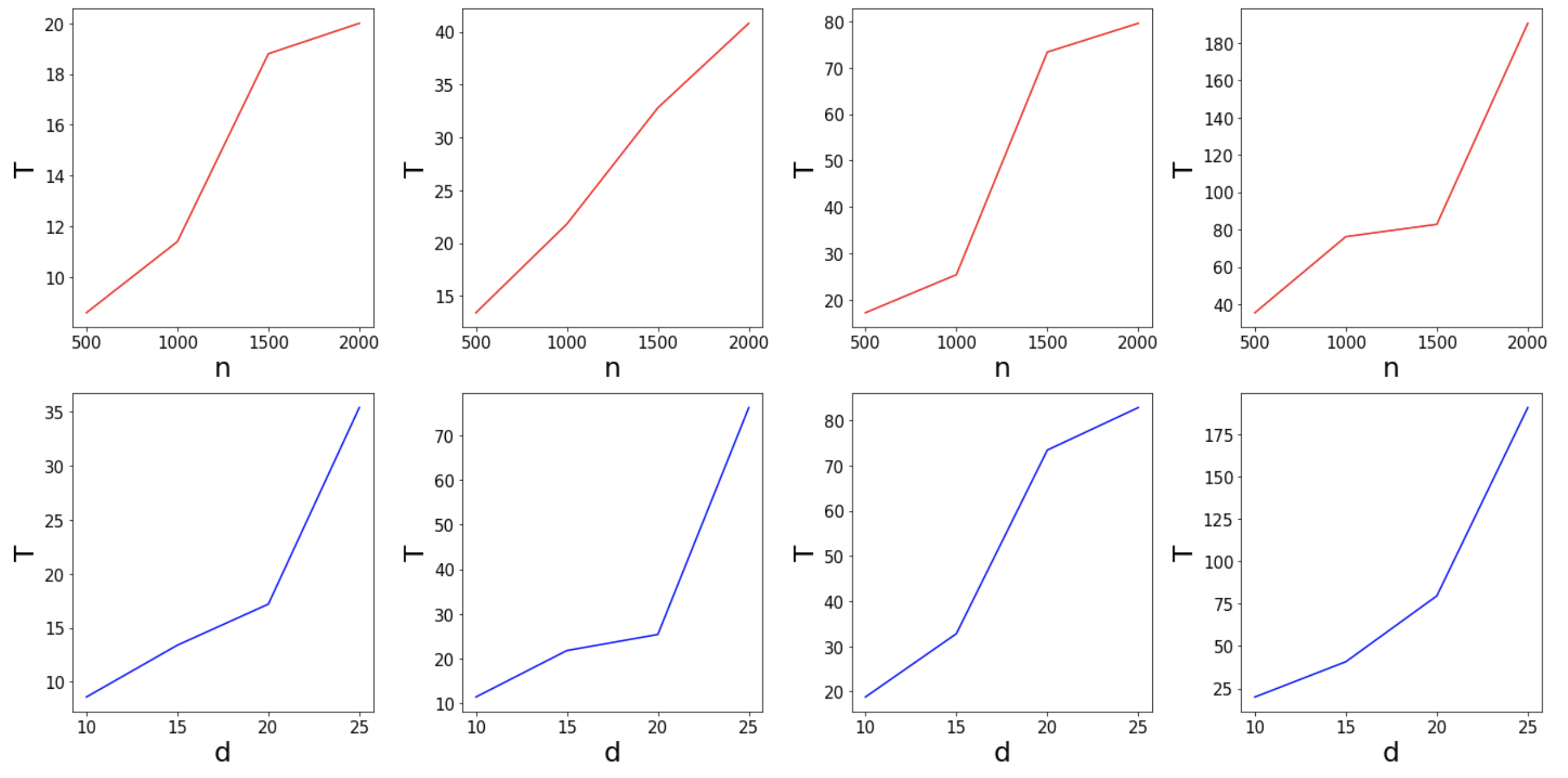}
    \caption{Experimental results presenting the number of samples required to reach \(\epsilon-\)suboptimality, with \(\epsilon=0.1\), for our algorithm \ref{OGDARR} on synthetic dataset with varying values of \(n\in\{500,1000,1500,2000\}\) and \(d\in\{10,15,20,25\}\).}
    \label{fig:DependenceOnN}
\end{figure}


\subsection{Logistic regression as a Generalized linear model} 
\label{subsec: Logistic Regression as Generalized linear model}





The goal in logistic regression is to maximize the log-likelihood of the conditional probability of \(y\) (the \emph{label}) given \(x\) (the \emph{feature}). In this model, it is assumed that:
\[
P(Y=1|x,\theta) = \frac{1}{1+\exp(-\lara{x,\theta})}
\]
This implies that:
\[
P(Y=-1|x,\theta) = \frac{\exp(-\lara{x,\theta})}{1+\exp(-\lara{x,\theta})}
\]

Given a data point \((x,y)\) the logistic loss is log-likelihood of  observing \(y\) given \(x\). For any \(\theta\) and \(y\in\{-1,1\}\):
\[
P(Y=y|x;\theta)  = \lr{P(Y=1|x,\theta)}^{\frac{1+y}{2}} \lr{P(Y=-1|x,\theta)}^{\frac{1-y}{2}}
\]

Now, the log-likelihood is given by:
\begin{align*}
L(x,y;\theta) &= \log(P(Y=y|x;\theta)) \\
&=  \frac{1+y}{2} \log\lr{\frac{1}{1+\exp(-\lara{x,\theta})}} +  \frac{1-y}{2}\log\lr{\frac{\exp(-\lara{x,\theta})}{1+\exp(-\lara{x,\theta})}} \\
&= -\frac{1-y}{2}\lara{x,\theta} + \lr{\frac{1+y}{2}+\frac{1-y}{2}} \log\lr{\frac{1}{1+\exp(-\lara{x,\theta})}} \\ 
&= -\frac{1-y}{2}\lara{x,\theta}  - \log(1+\exp(-\lara{x,\theta})) \\ 
&= \frac{y}{2}\lara{x,\theta} -\frac{1}{2}\lara{x,\theta} + \lara{x,\theta} - \log(1+\exp(\lara{x,\theta})) \\ 
&= \frac{y}{2}\lara{x,\theta}+ \frac{1}{2}\lara{x,\theta}- \log(1+\exp(\lara{x,\theta}))
\end{align*}

The goal is to maximize the log likelihood, which is equivalent to minimizing the negative log likelihood. Thus the logistic regression minimizes the following loss:
\[
\td{L}(x,y;\theta) = -L(x,y;\theta) = - \frac{y}{2}\lara{x,\theta} + \phi(\lara{x,\theta})
\]
where \(\phi(\beta) = \log(1+\exp(\beta))-\frac{\beta}{2}\). If \(y = 1\), then the above loss becomes:
\[
\log(1+\exp(\lara{x,\theta})) - \lara{x,\theta} = \log(1+\exp(-\lara{x,\theta}))
\]
Otherwise, if \(y = -1\), then the above loss becomes $\log(1+\exp(\lara{x,\theta}))$. Thus, the above loss is equivalent to \(\log(1+\exp(-y\lara{x,\theta}))\). 

\newpage

\end{document}